\documentclass[11pt]{amsart}

\usepackage{graphicx}

\usepackage[utf8]{inputenc}	

\usepackage[all]{xy}
\usepackage{pb-diagram, pb-xy}

\usepackage{amssymb}
\usepackage{amsthm}

\usepackage{url}

\usepackage{tikz}
\usepackage{pgffor}

\usepackage{hyperref}

\usepackage{calc}

\usepackage[framemethod=tikz]{mdframed}

\usepackage{enumitem}
\setlist[itemize]{noitemsep}

\usepackage{dsfont}

\usepackage{float}
\restylefloat{table}

\newtheoremstyle{saetze} 
    {5pt}                    
    {5pt}                    
    {\itshape}                   
    {12pt}                           
    {\bfseries}                   
    {.}                          
    {.5em}                       
    {}  

\theoremstyle{saetze}
\newtheorem{theorem}{Theorem}[section]
\newtheorem{lemma}[theorem]{Lemma}
\newtheorem{corollary}[theorem]{Corollary}
\newtheorem{proposition}[theorem]{Proposition}
\newtheorem{question}[theorem]{Question}

\theoremstyle{definition}
\newtheorem{definition}[theorem]{Definition}
\newtheorem{example}[theorem]{Example}
\newtheorem{remark}[theorem]{Remark}

\newcommand{\Q}{{\mathbb Q}}

\newcommand{\End}{{\mathit{End}}}
\newcommand{\Hom}{{\mathit{Hom}}}

\newcommand{\id}{{\mathit{id}\hspace{-0.1em}}}

\newcommand{\one}{{\mathbf{1}}}

\newcommand{\Z}{{\mathbf{Z}}}

\def\N{\mathbb{N}}

\DeclareMathOperator*{\bigdoublevee}{\bigvee\mkern-15mu\bigvee}

\newcommand{\sVW}{\mathit{s}\hspace{-0.7mm}\bigdoublevee}
\DeclareMathOperator{\tr}{t}

\newcommand{\mt}{\operatorname{\mathsf{t}}}

\renewcommand{\dim}{\ensuremath{\operatorname{dim}}}

\hypersetup{
    colorlinks,
    citecolor=red,
    filecolor=black,
    linkcolor=blue,
    urlcolor=black
}

\long\def\comment#1{}

\def\ignore#1{\relax}

\def\h{\mathfrak h}
\def\sl{\mathfrak{sl}}

\def\R{{\mathbb R}}
\def\Z{{\mathbb Z}}

\def\Q{{\mathbb Q}}
\def\Co{{\mathbb C}}

\def\la{\lambda}

\def\al{\alpha}

\def\X{{\bf X}}
\def\End{{\rm End}}
\def\1{{\bf 1}}
\def\Hom{{\rm Hom}}

\def\N{\mathcal N}

\def\C{\mathcal C}
\def\L{\mathcal L}
\def\I{\mathcal I}

\def\one{\mathbf 1}

\def\Fr{F^{(r)}}
\def\lar{\la^{(r)}}

\def\hstar{\h^*}

\def\Trk{Tr^{(k)}}
\def\dimk{\dim^{(k)}}

\begin{document}

\title{Generalized negligible morphisms and their tensor ideals}
\author{{\rm Thorsten Heidersdorf and Hans Wenzl}}

\address{T. H.: Max-Planck Institut f\"ur Mathematik, Bonn \\ Mathematisches Institut, Universit\"at Bonn}
\email{heidersdorf.thorsten@gmail.com} 
\address{H. W.: Department of Mathematics\\ University of California\\ San Diego,
California}
\email{hwenzl@ucsd.edu}

\date{}

\begin{abstract} We introduce a generalization of the notion of a negligible morphism and study the associated tensor ideals and thick ideals. These ideals are defined by considering deformations of a given monoidal category $\mathcal{C}$ over a local ring $R$. If the maximal ideal of $R$ is generated by a single element, we show that any thick ideal of $\mathcal{C}$ admits an explicitly given modified trace function. As examples we consider various Deligne categories and the categories of tilting modules for a quantum group at a root of unity and for a semisimple, simply connected algebraic group in prime characteristic. We prove an  elementary geometric description of the thick ideals in quantum type A and propose a similar one in the modular case.
\end{abstract}

\subjclass[2010]{17B10, 17B37, 18D10, 19D23.}

\keywords{Tensor ideals, Monoidal categories, Quantum groups, Local rings, Deligne categories, Tilting modules, Categorial dimensions, Modified traces, Weyl groups, Kazhdan-Lusztig cells}

\maketitle

\setcounter{tocdepth}{1}

\section{Introduction}

\subsection{Negligible morphisms} We are interested in the structure of tensor ideals \cite{Coulembier} \cite{Balmer} \cite{Boe-Kujawa-Nakano-tensor}  \cite{Comes-Heidersdorf} in a rigid spherical monoidal category $\mathcal{C}$ over a field $\Bbbk$. Roughly speaking we have two notions of tensor ideals:
\begin{itemize}
\item {\it tensor ideals}, submodules $\mathcal{I}(X,Y) \subset Hom_{\mathcal{C}}(X,Y)$ (for any $X,Y \in \mathcal{C}$) that are closed under composition and tensor products with morphisms;
\item {\it thick (tensor) ideals}, subsets of $ob(\mathcal{C})$ which are closed under tensor products with arbitrary objects of $\mathcal{C}$ and also closed under retracts.
\end{itemize}

Arguably the most important tensor ideal is the tensor ideal of negligible morphisms \[ \mathcal{N}(X,Y) = \{ f:X \to Y \ |  \ Tr(f \circ g) = 0  \ \ \forall g:Y \to X \}.\] By \cite{Andre-Kahn} it is the largest proper tensor ideal of $\mathcal{C}$, and the only tensor ideal such that the categorial quotient $\mathcal{C}/\mathcal{N}$ can be semisimple. Its associated thick ideal \[ N = \{ X \in \mathcal{C} \ |  \ X \cong 0 \in \mathcal{C}/\mathcal{N} \} \] consists of direct sums of indecomposable objects $X$ whose categorial dimension $\dim(X) = 0$ vanishes (the negligible objects).

\subsection{Generalized negligible morphisms} It is the aim of this article to present a generalization of the notion of negligible morphism which will lead to a measure - the {\it nullity } - for the negligibility of an object $X \in \mathcal{C}$. While $\mathcal{N}$ and $\dim(X)$ can always be defined under our assumptions on $\mathcal{C}$, the definition of the generalized negligible tensor ideals $\mathcal{N}_I$ and the nullity requires more restrictive conditions.

In order to define these tensor ideals, we consider a deformation or lift of $\mathcal{C}$ to a monoidal category $\mathcal{C}_R$ over a local ring $R$ as follows. Let $\mathcal{C}_R$ denote a monoidal rigid spherical category whose $\Hom$ spaces are free $R$-modules satisfying $End(\one) = R$. For any ideal $I$ of $R$ we put \[\mathcal{N}_I (X,Y) = \{ f: X\to Y \ |  \  Tr_X(g\circ f)\in I \text{ and }Tr_Y(f\circ g)\in I \ \forall g: Y\to X \}, \] the tensor ideal $I$-negligible morphisms in $\mathcal{C}_R$. An object $X \in \mathcal{C}_R$ is called $I$-negligible if $Tr_X(a)\in I$
for all $a\in \End(X)$. For a fixed ideal $I$ we can also say that $f:X \to Y$ is $k$-negligible with respect to $I$ if $Tr(f \circ g)$ is in $I^k$. We will only use this in the special situation where $I = \mathfrak{m}$, the maximal ideal of $R$. We then obtain the { $k$-\it negligible morphisms} $\mathcal{N}_k := \mathcal{N}_{\mathfrak{m}^k}$ which form a decreasing chain \[ \mathcal{N}_1 \supset \mathcal{N}_2  \supset \mathcal{N}_3 \supset \ldots \] of tensor ideals in $\mathcal{C}_R$ and likewise for the $N_k$. 

In order to define these for the category $\mathcal{C}$ over $\Bbbk$ we suppose now that we have a surjective and full tensor functor $\mathcal{C}_R \to \mathcal{C}$ with $\Bbbk = R/\mathfrak{m}$. A special case of this is if $\mathcal{C}$ is the mod $\mathfrak{m}$ evaluation of $\mathcal{C}_R$: here the objects of $\mathcal{C}$ are the same as those of $\mathcal{C}_R$ with \[  Hom_{\mathcal{C}} (X,Y) = Hom_{\mathcal{C}_R} (X,Y)/ \mathfrak{m} Hom_{\mathcal{C}_R} (X,Y) \cong Hom_{\mathcal{C}_R} (X,Y) \otimes_R R/\mathfrak{m}.\]

When passing from $\mathcal{C}_R$ to $\mathcal{C}$, the images of the $\mathcal{N}_I$ and $N_I$ define tensor ideals and thick ideals in $\mathcal{C}$ respectively (possibly zero if we are not in the special situation of the mod $\mathfrak{m}$ evaluation) which we denote again by $\mathcal{N}_I$ and $N_I$ (or $\mathcal{N}_k$ and $N_k$). Note that $\mathcal{N}_1$ and $N_1$ are mapped to $\mathcal{N}$ and $N$ in $\mathcal{C}$. We call an indecomposable object of $\mathcal{C}$ $k$-negligible if $X \in N_k$. Its {\it nullity} is the smallest $k$ such that $X \in N_k$. 



\subsection{Examples}

Given a monoidal category $\mathcal{C}$ the question is whether it admits a lift to a monoidal category $\mathcal{C}_R$. Our main examples are the following:

\begin{theorem} \label{theorem-intro-eval} The following categories can be obtained as mod $\mathfrak{m}$ evaluations:
\begin{enumerate}
\item The category of (quantum) tilting modules $Tilt(U_q(\mathfrak{g}), \Q(q))$, where $\mathfrak{g}$ is a semisimple complex Lie algebra and $q$ a nontrivial primitive $\ell$-th root of unity with $\ell$ odd, $\ell > h$ and not divible by $3$ if $\mathfrak{g}$ contains $\mathfrak{g}_2$, is the mod $\mathfrak{m}$ evaluation of $Tilt(U_v(\mathfrak{g}),R)$ where $R$ is the completion of $\mathbb{Q}[v]_{(v-q)}$, the polynomial ring localized at $(v-q)$, i.e. all rational
functions over $\mathbb{Q}$ which are evaluable at $v=q$.
\item The category of (modular) tilting modules $Tilt(G,k)$, where $G$ is a semisimple simply connected algebraic group over a perfect field $\Bbbk$ of characteristic $p > 0$, is the mod $\mathfrak{m}$ evaluation of $Tilt(G,W(k))$ where $W(k)$ is the ring of Witt vectors of $\Bbbk$.
\item The Deligne categories $Rep(S_t)$, $Rep(GL_t)$ and $Rep(O_t)$, $t \in \Co$, are mod $\mathfrak{m}$ evaluations from their analogs over the completion of $\Co[t]_{(t-n)}$.
\end{enumerate}
\end{theorem}

The proof of theorem \ref{theorem-intro-eval} 1) and 2) is explained in sections \ref{tilting-quantum} - \ref{sec:tilting-eval}. The field $\mathbb{Q}(q)$ can of course be replaced by $\Co$. Part 3) can be generalized to incorporate the $q$-versions $Rep(U_q(\mathfrak{gl}_t))$ and $Rep(U_q(\mathfrak{o}_t))$ \cite{Heidersdorf-Wenzl-Deligne}. In this case however we deal with two parameter versions, e.g. the local ring is the completion of $\Co[r,\xi]_{r-\xi^{n-1}, \xi-q}$ (using BMW notation) for the $q$-version of $O_t$. Using 3) it is possible to define $k$-negligible ideals also for certain categories of representations of supergroups. In this case the nullity is related to the atypicality.



\subsection{Modified dimensions and link invariants} In recent years there has been a lot of interest in the construction of modified traces and dimension functions \cite{Geer-Kujawa-Patureau-Mirand} \cite{Geer-Kujawa-Patureau-Mirand-ambidextrous} \cite{Geer-Patureau-Mirand-Virelizer} \cite{Geer-Patureau-Mirand-projective}. One of the main motivations for the introduction of these modified trace and dimension functions is the construction of knot invariants since the invariant (in the sense of Reshetikhin-Turaev) of an indecomposable object $X$ with $\dim(X) = 0$ vanishes. It is very difficult to show that a given ideal has a nontrivial modified trace function. In many cases the only known thick ideal to admit such a nontrivial modified trace is the ideal of projective objects in $\mathcal{C}$, the smallest nontrivial  thick ideal. Moreover, in most cases these modified trace functions are not explicitely given.

The situation changes if there is a surjective and full tensor functor $\mathcal{C}_R \to \mathcal{C}$. In this case we can often renormalize the usual trace in $\mathcal{C}_R$ and consider its image in $\mathcal{C}$ to get a modified trace function.

\begin{theorem} Let $R$ be a local domain (which is not a field) whose maximal ideal $(p)$ is 
generated by the element $p$. Let $\C_R$ be a rigid spherical monoidal category
whose Hom spaces are free $R$-modules.
Let $I$ be a thick ideal all of whose objects are $k$-negligible (with respect to $(p)$),
such as e.g. the ideal $N_k$ of all $k$-negligible objects. For $X \in I$ and $a \in End(X)$ 
\begin{equation}
\Trk_X(a)\ :=\ \frac{1}{p^k}\ Tr_X(a),\hskip 3em \dimk(X)\ :=\ \frac{1}{p^k}\ \dim(X),
\end{equation}
define modified trace and dimension functions on $I$. The image of the modified trace function under $\C_R \to \C$ defines a modified trace function on the image of $I$ in $\C$.
\end{theorem}

The proof of this theorem is essentially trivial. The whole difficulty lies in the construction of an appropriate lift of $\mathcal{C}$ to an analogous category over a local ring $R$. Since any proper thick ideal is contained in the ideal of negligible objects, we obtain 

\begin{corollary} Under the assumptions of the theorem, every thick ideal in $\mathcal{C}$ admits a modified trace function. 
\end{corollary}

Since in each case in theorem \ref{theorem-intro-eval} the maximal ideal is generated by a single element, we obtain

\begin{corollary} Each thick ideal of $Tilt(U_q(\mathfrak{g}), \Q(q))$, $Tilt(G,k)$ and the Deligne categories $Rep(S_t)$, $Rep(GL_t)$ and $Rep(O_t)$, $t \in \Co$, admits a nontrivial modified trace function.
\end{corollary} 


The most interesting example for this is the case of $Tilt(U_q(\mathfrak{g}),\Co)$. The classical way of Reshetikhin-Turaev to define link invariants colored by objects of $Tilt(U_q(\mathfrak{g}),\Co)$ yields a trivial invariant $\L$ unless the objects are all in the fundamental alcove. Due to our lifting theorem, we can directly define a link invariant in the sense of Reshetikhin-Turaev over the local ring $R$, the completion of $\Co[v]_{(v-q)}$. This invariant can be normalized by $\frac{1}{p^k}$ like the dimension function and yields an $R/(p)$ valued invariant for $Tilt(U_q(\mathfrak{g}),\Co)$.

\begin{theorem} (see theorem \ref{thm:mod-link}) Assume that the components of the link have been colored with the objects $X_1,\ldots, X_m \in \mathcal{C}_R$. Let $k$ be the nullity of $X_1^{\otimes c_1}\otimes X_2^{\otimes c_2}\otimes\ ...
\otimes X_m^{\otimes c_m}$. Then the value of the $R/(p)$-valued invariant
is equal to $\frac{1}{k!}\frac{d^k}{dv^k}\L^{(X_1,\ ..., X_m)}(L)_{|v=q}$,
which is valid for its evaluation on any $m$-component link $L$.
\end{theorem}

\subsection{Thick ideals for tilting modules}

What does the nullity capture in the tilting module case? 

\begin{itemize}
\item In the modular case, the maximal ideal $\mathfrak{m}$ of the complete discrete valuation ring $W(k)$ is generated by $p$. If a tilting module over $W(k)$ is in $N_k$, then in particular $rank(T(\lambda)) \in (p)^k$. There are however tilting modules over $W(k)$ satisfying $p^k | rank_{W(k)} T(\lambda)$ while $T(\lambda) \notin N_k$. However, if $T(\lambda)$ is irreducible, then $T(\lambda) \in N_k$ if and only if $p^k | rank_{W(k)} T(\lambda)$.
\item In the quantum case the maximal ideal is generated by $(v-q)$. The dimension is an element in the subring $\mathbb{Q}[v]_{(v-q)}$; and the dimension over $\Q(q)$ is obtained by evaluating this rational function at $v=q$. As for the modular case, $T(\lambda) \in N_k$ implies that the  multiplicity of $(v-q)^k$ in the numerator is at least $k$; and conversely if $T(\lambda)$ is irreducible.
\end{itemize}

In the modular and quantum case the thick ideals are sums of thick ideals attached to a right $p$-cell or a right cell in the affine Weyl group $W^+_{p}$ (or $W^+_{\ell}$) (see section \ref{sec:classif}) by results of \cite{Achar-Hardesty-Riche} \cite{Ostrik-ideals}. The combinatorics of these cells is however very difficult and not fully understood, especially in the modular case \cite{Jensen-Phd}. We construct thick ideals $\mathcal{I}(F)$ associated to minimal facets $F$ and compute their nullities in Proposition \ref{tensorexample}. This suggests a description of tensor ideals as a collection of positive cones associated to certain facets.

For the quantum type $A_{n-1}$ every thick ideal is a sum of thick ideals attached to Young diagrams $\lambda$ of size $n$ (which parametrize the two-sided cells of the affine Weyl group). We attach a {\it standard facet} $F_0(\lambda)$ to every such Young diagram and prove:

\begin{theorem} (see theorem \ref{TypeAmain} for details) 
The thick ideal $\I(\la)=\I(F_0(\la))$
generated by  the tilting modules $T(\nu)$ for which
$\nu+\rho\in F_0(\la)$ coincides with the thick ideal constructed by
Ostrik for the cell in the dominant Weyl chamber corresponding to the two-sided
cell labeled by the Young diagram $\la^T$.  In particular, the nullity of any
generating module $T(\nu)$ of that ideal is equal to the value of Lusztig's $a$-function of that cell.
\end{theorem}

For the relation between $N_k$ and the values of the $a$-function in all types see remark \ref{a-fun}. For type $A$
we can also give the already alluded geometric description of tensor ideals via positive cones associated to certain
facets, see Theorem \ref{Ilambda} for details. This follows fairly easily from  earlier work of \cite{Shi}, 
where we have benefitted from its description in 
\cite{Cooper}. In particular, we obtain an explicit description of the $N_k$. Moreover, this approach also
suggests a description of
the ideal structure for the modular case which is done in section \ref{sec:modular-regions}.

\subsection{Structure of the article}

In section \ref{sec:k-negl} we introduce basic properties of the generalized negligible ideals. Modified trace functions are studied in section \ref{sec:traces} and modified link invariants in section \ref{sec:link}. Sections \ref{tilting-quantum} - \ref{sec:tilting-eval} deal with the case of tilting modules. In section \ref{sec:typeA} we give a description of the thick ideals in quantum type $A$. We end the article with some open questions in section \ref{sec:questions}. A second article \cite{Heidersdorf-Wenzl-Deligne} will treat the case of Deligne categories. A third article will deal with open questions about the thick ideals and the $N_k$ for modular and quantum tilting modules.



\section{$k$-negligible morphisms and their tensor ideals} \label{sec:k-negl}

\subsection{Preliminaries} \label{sec:prel} In the following let $R$ be a local ring
with maximal ideal $\mathfrak{m}$. We assume $\C_R$ (or sometimes simply $\C$) to be a monoidal
rigid spherical category whose $\Hom$ spaces are free $R$-modules and such that $End(\one) = R$ (see e.g. \cite[Section 4.7]{EGNO} for details). For simplicity we additionally assume that $\C_R$ is braided. Otherwise we would have to distinguish between left thick ideals and right thick ideals and between partial and modified traces for the left and right versions. However, the followings notions make sense without the added \emph{braided} if one is willing to either work with left or right versions of these. 

Recall that under these assumptions there exist, for each object $X$ in $\C_R$, canonical morphisms
$$i_X: \1\to X\otimes X^*,\quad \tilde d_X: X\otimes X^* \to \1$$
via which we can define the {\it trace} $Tr_X$ on $\End(X)$ by
$$Tr_X(a)=\tilde d_X (a\otimes id_{X^*})i_X,\quad {\rm for\ all}\ a\in \End(X);$$
here $i_X$ and $d_{X^*}: X^{**}\otimes X^*\to\1$ are the morphisms 
in the definition of rigidity for the objects $X$ and $X^*$, and
$\tilde d_X= d_{X^*}\circ (s_X\otimes 1)$. The isomorphisms $s_X:X\to X^{**}$ from the spherical structure
are normalized such that $\dim(X)=\dim(X^*)$ for all objects $X$ in $\C$,
where $\dim(X)=Tr(id_X)$.
For elements $a\in \End(X\otimes Y)$, we can also define the {\it partial trace}
or {\it conditional expectation} $E_X:\End(X\otimes Y)\to \End(X)$ by
$$E_X(a)= (id_X\otimes \tilde d_Y)(a\otimes id_{Y^*})(id_X\otimes i_Y).$$
The name {\it partial trace} is justified by the equation
$$Tr_X(E_X(a))=Tr_{X\otimes Y}(a),\quad a\in \End(X\otimes Y).$$

\medskip\noindent
\subsection{Tensor ideals} Let $\mathcal{C}$ be a monoidal category. A {\it tensor ideal} $\I$ in $\mathcal{C}$ consists of an $R$-submodule $\I(X,Y) \subset Hom(X,Y)$ for all $X,Y \in \C$ such that 
\begin{itemize} 
\item for all $X,Y,Z,W \in \C$ and $g \in Hom(X,Y)$ and $h \in Hom(Z,W)$ \[ f \in \I(Y,Z) \text{ implies } f \circ g \in \I(X,Z) \text{ and } h \circ f \in \I(Y,W); \] 
\item $f \in \I(X,Y)$ implies $id_Z \otimes f \in \I(Z \otimes X,Z \otimes Y)$ and likewise from the right.
\end{itemize}

A collection of objects $I$ in a monoidal category $\C$ is called a \emph{thick ideal} of $\C$ if the following conditions are satisfied:
\begin{enumerate}
\item[(i)] $X\otimes Y\in I$ whenever $X\in\C$ and $Y\in I$.
\item[(ii)] If $X\in\C$, $Y\in I$ and there exist $\alpha:X\to Y$, $\beta:Y\to X$ such that $\beta\circ\alpha=\id_X$, then $X\in I$.
\end{enumerate} 

To any tensor ideal $\I$ we can associate the thick ideal $I$ given by \[ I = \{ X \in \C \ | \ id_X \in \I(X,X) \}.\] 

One of the major reasons to study the tensor ideals and thick ideals in $\C$ is due to the fact that the morphisms that are sent to zero under a monoidal functor $\C \to \C'$ to another monoidal category $\mathcal{C}'$ form a tensor ideal; and the objects of $\C$ that are sent to zero form a thick ideal.

\subsection{Generalized negligible morphisms} Let $R$ be a local ring and $\C_R$ as in section \ref{sec:prel}.

\begin{definition} a) Let $I \subset R$ be an ideal. We call a morphism $f: X\to Y$ $I$-negligible
if $Tr_X(g\circ f)\in I$ and $Tr_Y(f\circ g)\in I$ for all morphisms
$g: Y\to X$. An object $X$ is called $I$-negligible if $Tr_X(a)\in I$
for all $a\in \End(X)$.

b) If $f$ is $I$-negligible with respect to $I = \mathfrak{m}^k$, we simply say that $f$ is $k$-negligible. An object is $k$-negligible if it is $I$-negligible for $I = \mathfrak{m}^k$.
\end{definition}

\begin{lemma} The $I$-negligible morphisms form a tensor ideal $\N_I$ in the category $\C_R$. The $I$-negligible objects form a thick ideal $N_I$ in $\C_R$.
\end{lemma}

\begin{proof} It is easy to see that $\mathcal{N}_I$ is an ideal using $Tr(f\circ g) = Tr(g \circ f)$ for composable morphisms $f:X \to Y$, $g:Y \to X$ \cite[Theorem XIV.4.2]{Kassel}. Now let $f \in \mathcal{N}_I(X,Y)$ and $g \in Hom(W,Z)$ arbitrary for $X,Y,W,Z$ in $\mathcal{C}_R$. Let $h \in Hom(Y \otimes Z,X \otimes W)$. Then \[ Tr((f \otimes g) \circ h) = Tr(h \circ (f \otimes g)) = Tr(h' \circ f)\] for some $h': Y  \to X$ (as in \cite[Theorem 2.9]{Barrett-Westbury}). Since $f$ is $I$-negligible, this implies $Tr((f \otimes g) \circ h) \in I$.

Let now $X$ be an $I$-negligible object, and $Y$ any object in $\C$.
If $a\in \End(X\otimes Y)$, we have
$$Tr_{X\otimes Y}(a)=Tr_X(E_X(a))\in I,$$
as $E_X(a)\in\End(X)$ and $X$ is $I$-negligible. Hence $X\otimes Y$ is an object in $N_I$ as well.
\end{proof}

\begin{remark} The definition is similar to the the one of the Jantzen filtration on morphisms defined by the form $Tr(f \circ g)$.
\end{remark}

\subsection{The mod $\mathfrak{m}$ evaluation} We are primarily interested in categories
whose $\Hom$ spaces are vector spaces over a field. Recall that if $M$ is a
free $R$-module of rank $r$, we obtain a well-defined vector space
$M/\mathfrak{m}M$ over $\Bbbk=R/\mathfrak{m}$ of dimension $r$.  We call the {\it mod $\mathfrak{m}$
evaluation} (or reduction modulo $\mathfrak{m}$) of $\C_R$ the category $\C$ over $\Bbbk$ whose objects are in 1-1
correspondence with the ones of $\C_R$, and where \[ Hom_{\C}(X,Y)=Hom_{\C_R}(X,Y)/\mathfrak{m}Hom_{\C_R}(X,Y) \cong Hom_{\C_R}(X,Y) \otimes_R R/\mathfrak{m}.\]
In the following, the notations $\Hom$, $\End$ etc will refer to the evaluation
category $\C$. The corresponding spaces for $\C_R$ will be denoted by $\Hom_R$,
$\End_R$ etc. We call $\C_R$ a {\it lift} of $\C$.

\subsection{Examples} We give some examples of the lifting of a monoidal category $\mathcal{C }$ over $\Bbbk$ to a monoidal category over a local ring.

\subsubsection{Fusion categories} 

Let $\Bbbk$ be any field, $R$ a local ring with $R/\mathfrak{m} \cong k$. If $\mathcal{C}$ is a split fusion category over $\Bbbk$, a lifting of $\mathcal{C}$ in the sense of \cite[9.16]{EGNO} is a split fusion category $\tilde{\mathcal{C}}$ over $R$ such that $\mathcal{C}$ is the mod $\mathfrak{m}$ evaluation of $\tilde{\mathcal{C}}$.

\begin{theorem} \cite[Theorem 9.16.1]{EGNO} If the global dimension of $\mathcal{C}$ is non-zero, $\mathcal{C}$ admits a lifting to $R$ and this lifting is unique up to equivalence.
\end{theorem}

Of particular interest is the situation where $\Bbbk$ is a perfect field of prime characteristic $p$. Then the ring of Witt vectors $W(k)$ is a discrete valuation ring with maximal ideal generated by $p$ and $W(k)/pW(k) \cong k$. If $\mathcal{C}$ is a non-degenerate (symmetric/braided) fusion category category over $\Bbbk$, then it admits a (symmetric/braided) lifting to $W(k)$ by \cite[Theorem 9.3, Corollary 9.4]{ENO-fusion}. Since we are interested here in the construction of tensor ideals, the semisimple case is not relevant to us.

\subsubsection{Algebraic groups} \label{sec:alg}
While we can define the mod $\mathfrak{m}$ evaluation for any monoidal category over the local ring $R$, it is often not the correct category one is interested in. Consider the case of an algebraic group $G$ over the local ring $R$. Then extension of scalars of $Rep(G,R)$ defines a monoidal functor \begin{align*} Rep(G,R) & \to Rep(G \otimes R/\mathfrak{m}, R/\mathfrak{m}) \\ V & \mapsto V \otimes_R R/\mathfrak{m} \\ Hom_G(X,Y) & \mapsto Hom_G(X,Y) \otimes R/\mathfrak{m}.\end{align*} The image of $Rep(G,R)$ under this functor is the mod $\mathfrak{m}$ evaluation, but it is not the category $Rep(G \otimes R/\mathfrak{m}, R/\mathfrak{m})$ (unless we are in the semisimple case). Indeed the canonical functor \[ Hom_G (M,N) \otimes R/\mathfrak{m} \to Hom_{G_{R/\mathfrak{m}}}(M \otimes R/\mathfrak{m}, N \otimes R/\mathfrak{m}),  \] where $G_{R/\mathfrak{m}}$ is the algebraic group over $\Bbbk$ obtained by extension of scalars, is in general not bijective \cite[10.14]{Jantzen}.

\subsubsection{Deligne categories} For every field $\Bbbk$ Deligne \cite{Deligne-interpolation} defined symmetric monoidal categories $Rep(S_t)$, $Rep(GL_t)$ and $Rep(O_t)$, $t \in k$, which interpolate the representation categories of the symmetric group, the general linear group and the orthogonal group. Each of this categories is constructed in the following way: One defines a skeletal subcategory corresponding to the tensor powers of the permutation representation $V$ of $S_n$, the standard representation $V$ of $O(n)$ or the tensor product $V \otimes V^{\vee}$ of the standard representation $V$ of $GL(n)$ and its dual. The object corresponding to such a tensor power $V^{\otimes r}$ is denoted $r$ in the $S_n$ and $O(n)$-case and $(r,s)$ in the $GL(n)$-case. The endomorphism algebras of these objects are by definition

\begin{enumerate}
\item $End_{Rep(S_t)}(r) = kP_r(t)$, the partition algebra for the parameter $t$.
\item $End_{Rep(O_t)}(r) = kBr_r(t)$, the Brauer algebra for the parameter $t$.
\item $End_{Rep(GL_t)}(r,s) = kWB_{r,s}(t)$, the walled Brauer algebra for the parameter $t$.
\end{enumerate}

To get the full category, we take the additive karoubian envelope of the skeletal subcategory. The categories $Rep(S_t)$, $Rep(GL_t)$ and $Rep(O_t)$ admit a lift to the completion of the local ring of evaluable rational functions $R = k[T]_{(T-t)}$ \cite{Heidersdorf-Wenzl-Deligne}. Indeed the construction described above makes sense over $R$ as well. This can be generalized to inlcude the $q$-deformations of $Rep(O_t)$ and $Rep(GL_t)$.

\subsubsection{Tilting modules} Let $Tilt$ denote the monoidal category of modular/quantized tilting modules. Then $Tilt$ admits a lift to the category of tilting modules over the ring of Witt vectors $W(k)$ (where $\Bbbk$ is a perfect field of characteristic $p > 0$) respectively a lift to the category of tilting modules over the completion of $\Q[v]_{v-q}$. For details we refer the reader to sections \ref{tilting-modular}, \ref{tilting-quantum} and \ref{sec:tilting-eval}.


\subsection{New tensor ideals}

\begin{lemma} The tensor ideals $\mathcal{N_I}$ of $\mathcal{C}_R$ define tensor ideals in the mod $\mathfrak{m}$ evaluation $\mathcal{C}$. The thick ideals $N_I$ define thick ideals in the mod $\mathfrak{m}$ evaluation. The tensor ideal $\mathcal{N}_1$ corresponds to the ideal of negligible morphisms in $\mathcal{C}$ and the thick ideal $N_1$ corresponds to the indecomposable objects of categorial dimension 0.
\end{lemma}

In particular we obtain a chain of tensor ideals \[  \ldots \subseteq \mathcal{N}_3 \subseteq \mathcal{N}_2 \subseteq \mathcal{N}\] and likewise a chain of thick ideals \[  \ldots \subseteq N_3 \subseteq N_2 \subseteq N.\] 

The number $k$ can often be seen as a measure for the vanishing of $\dim(X)$. If $X \in N_k$ with $k$ minimal, we say that $X$ has {\it nullity} $k$. For the explicit meaning of this nullity we refer to the  examples that appear later in the article.

\begin{question} The following question was raised by Kevin Coulembier and Victor Ostrik: Can one find a local ring $R$ such that the $N_J$, where $J$ runs over the ideals of $R$, is a complete list of thick ideals in $\mathcal{C}$? While we do not know the answer, the existence of a lifting to a local ring seems delicate in the non-semisimple case.
\end{question} 


\subsection{Compatibilities and $k$-semisimplicity}

The following lemma follows immediately from the definition.

\begin{lemma} Let $\mathcal{C}$ be the mod-$\mathfrak{m}$ evaluation of $\mathcal{C}_R$ over the local ring $R$. The diagram \[ \xymatrix{ \mathcal{C}_R \ar[r] \ar[d] & \mathcal{C} \ar[d] \\ \mathcal{C}_R/\mathcal{N}_k \ar[r] & \mathcal{C}/\mathcal{N}_k } \] commutes. In particular $\mathcal{C}/\mathcal{N}_k$ is the mod $\mathfrak{m}$ evaluation of $\mathcal{C}_R/\mathcal{N}_k$. 
\end{lemma}

Since $\mathcal{N}_k \subset \mathcal{N}_{k+1}$ we obtain a chain of full and surjective tensor functors  \[ \xymatrix{ \ldots \ar[r] \ar[d] & \mathcal{C}_R/\mathcal{N}_3 \ar[r] \ar[d] & \mathcal{C}_R/\mathcal{N}_2 \ar[r] \ar[d] & \mathcal{C}/\mathcal{N} \ar[d] \\ \ldots \ar[r] & \mathcal{C}/\mathcal{N}_3 \ar[r] & \mathcal{C}/\mathcal{N}_2 \ar[r] & \mathcal{C}/\mathcal{N}.  } \] If $\C$ is a tensor category so that $\mathcal{C}/\mathcal{N}$ is semisimple, this loosely suggests to interpret $\mathcal{C}/\mathcal{N}_k$ as \emph{$k$-semisimple}. This can be made more precise by considering the trace. Recall \cite[Proposition 5.7]{Deligne-interpolation} that a tensor category is semisimple if and only if the trace pairing $Hom(X,Y) \times Hom(Y,X) \to k$ is non-degenerate, i.e. $Tr(fg) = 0$ for all $g:Y \to X$ implies $f=0$. For a morphism $f$ in $\mathcal{C}/\mathcal{N}_k$ the deviation for the failure of the non-degeneracy condition can be seen as an element in $I^{\leq k-1}$ by considering the lift of $f$ in $\mathcal{C}_R/\mathcal{N}_k$.




\subsection{A generalization} \label{sec:general} If a monoidal functor is full and surjective (on objects), the images of the tensor ideals and thick ideals $\mathcal{N}_k$ are again tensor ideals and thick ideals. Therefore we can more generally define $k$-negligible morphisms and objects provided we have a full and surjective monoidal functor $\mathcal{C}_R \to \mathcal{C}$.


\subsection{Monoidal supercategories}

The notion of an $I$-negligible ideal and the mod $\mathfrak{m}$ evaluation can be defined in the same way for monoidal supercategories as in \cite{Brundan-Ellis} \cite[A.1.2]{Coulembier}.  Examples of monoidal supercategories are the odd Temperley-Lieb supercategory $\mathcal{STL}(\delta), \delta \in k$ of \cite[Example 1.17]{Brundan-Ellis}, the affine VW supercategory $\sVW$ \cite{many-1} and the oriented Brauer-Clifford and degenerate affine oriented Brauer-Clifford supercategories \cite{Brundan-Comes-Kujawa}. An ideal $J$ in a supercategory is an ideal as in an ordinary category with the extra assumption that $J(X,Y)$ is a graded subgroup of $Hom(X,Y)$ for all $X,Y \in \mathcal{C}$. The notion of a tensor ideal is otherwise unchanged. Thick ideals can be defined as for monoidal categories.



\section{Modified traces and dimensions} \label{sec:traces}

We recall the definition of a modified trace function. The existence of such trace functions on thick ideals is in general difficult and often only known on the thick ideal of projective objects. We show that these exist for our rigid spherical category provided it admits a lift to a local ring whose maximal ideal is principal.

\subsection{The concept of a modified trace}


Let $\C_R$ be rigid spherical monoidal over a local ring $R$. As in Section \ref{sec:prel} we assume that $\C_R$ is braided in order to identify left and right duals. We follow \cite{Geer-Kujawa-Patureau-Mirand} \cite{Geer-Kujawa-Patureau-Mirand-ambidextrous} \cite{Geer-Patureau-Mirand-Virelizer} \cite{Geer-Patureau-Mirand-projective} in the definition of a trace function. Recall that for any objects $X,Y \in \C$ and any endomorphism $f \in End_\mathcal{C}(X \otimes Y)$ we have the left  trace $t_L(f) \in End_{\C}(X)$ and the right trace $t_R(f) \in End_{\C}(Y)$ defined as follows \begin{align*} tr_L(f) & = (d_X \otimes id_Y) \circ (id_{X^*} \otimes f) \circ (\tilde{i}_X \otimes id_Y) \in End_{\mathcal{C}} (X) \\ tr_R(f) & = (id_X \otimes \tilde{d}_Y) \circ (f \otimes id_{Y^*}) \circ (id_X \otimes i_{Y}) \in End_{\mathcal{C}}(Y).\end{align*}

\begin{definition}\label{D:trace}  If $I$ is a thick ideal in $\mathcal{C}$ then a \emph{trace on $I$} is a family of linear functions
$$\{\mt_V:\End_\mathcal{C}(V)\rightarrow R\}$$
where $V$ runs over all objects of $I$ and such that the following two conditions hold.
\begin{enumerate}
\item  If $X\in I$ and $Y\in \mathcal{C}$ then for any $f\in End_\mathcal{C}(X\otimes Y)$ we have
\[ 
\mt_{X\otimes Y}\left(f \right)=\mt_X \left( \tr_R(f)\right).\]
\item  If $X,Y\in I$ then for any morphisms $f:X\rightarrow Y $ and $g:Y\rightarrow X$  in $\mathcal{C}$ we have 
\[ 
\mt_X(g\circ f)=\mt_Y(f \circ g).\]

\end{enumerate}
\end{definition}

Given such a trace on a thick ideal $I$, $\{t_{V} \}_{V \in I}$, define the {\it modified dimension} function on objects of $I$ as the modified trace of the identity morphism:
\begin{equation*}
d\left(V \right) =t_{V}(id_{V}).
\end{equation*}

\subsection{Existence of modified traces} 

As was shown in \cite{Geer-Kujawa-Patureau-Mirand-ambidextrous}, modified trace functions exist on the thick ideal of projective objects in a number of examples. Beyond that, little seems to be known for general thick ideals.

\medskip\noindent
Let $R$ be a local domain (which is not a field) whose maximal ideal $(p)$ is 
generated by the element $p$. Let $I$ be a thick ideal all of whose objects are $k$-negligible (with respect to $(p)$),
such as e.g. the ideal $N_k$ of all $k$-negligible objects. Then we define
the modified trace $\Trk_X$ and modified dimension
$\dimk(X)$ for an object $X$ in $I$ by
\[
\Trk_X(a)\ =\ \frac{1}{p^k}\ Tr_X(a),\hskip 3em \dimk(X)\ =\ \frac{1}{p^k}\ \dim(X),
\]
where $a\in\End(X)$. Note that this is well-defined since $Tr_X(a) \in (p)^k \ \forall a \in End(X)$. It is clear that $\Trk_X(id_X)\ = \dimk(X)$.

We list some elementary 
properties of these modified traces.

\begin{lemma}\label{modifiedlem} Let $X,Y$ be  objects in $I$,
and let $Z$ be an object in $\C$. Then we have

(a) $\Trk_Y(ab)=\Trk_X(ba)$ for all morphisms $a:X\to Y$ and $b:Y\to X$,

(b) $\Trk_{X\otimes Z}(a\otimes c)=\Trk_X(a)Tr_Z(c)$ and
$\dimk(X\otimes Y)=\dimk(X)\dim(Y)$ for $a\in\End(X)$,
$c\in \End(Z)$.

(c) $\Trk_{X\otimes Y}(f) = \Trk_X(t_R(f))$ for all $f \in End(X\otimes Y)$.
\end{lemma}

\begin{proof} These properties hold for the ordinary trace. By our assumptions
 the renormalizations $Tr^{(k)}$ are also well-defined for the elements to which they are applied in our statements. As the renormalization factor is the same on both sides of the equation, the statements are also true for $Tr^{(k)}$.
\end{proof}

Taking the images of these modified traces defines modified trace functions $\Trk_X$ on $N_k \subset \C$. These modified traces are nontrivial on $N_k \setminus N_{k+1} \subset \C$ since there exists $a \in End_{\C_R}(X)$ satisfying $Tr_X(a) \in (p^k) \setminus (p^{k+1})$ for $X \in N_k \setminus N_{k+1} \subset \C_R$.

\begin{corollary} Suppose that $R$ is not a field and that $\mathfrak{m} = (p)$ is principal. Every thick ideal in a mod $\mathfrak{m}$ evaluation carries a nontrivial modified trace functions. 
\end{corollary}


\begin{corollary} Every thick ideal in the following categories admits a nontrivial modified trace.
\begin{enumerate}
\item The Deligne categories $Rep(S_t)$, $Rep(GL_t)$, $Rep(O_t)$, $t \in \Co$.
\item Let $Tilt(U_q(\mathfrak{g}), \Q(q))$ denote the category of tilting modules in the category of finite dimensional modules of Lusztig's quantum group where $\mathfrak{g}$ is a semisimple Lie algebra and $q$ a primitive $\ell$-th root of unity where $\ell > h$ and $\ell$ is not divisible by 3 if $\mathfrak{g}$ contains $\mathfrak{g}_2$.
\item Let $Tilt(G,k)$ denote the category of tilting modules in the category of finite dimensional representations of $G$, where $G$ is a semisimple and simply connected algebraic group over a perfect field $\Bbbk$ of characteristic $p >0$. 
\end{enumerate}
\end{corollary}

\begin{proof} Each category is obtained as a mod $\mathfrak{m}$ reduction of a monoidal category over a discrete valuation ring (see theorem \ref{thm:tilting-eval} and \cite{Heidersdorf-Wenzl-Deligne}).
\end{proof} 

\begin{remark} The results of this section generalize to the situation of section \ref{sec:general} where we have a full and surjective monoidal functor $\mathcal{C}_R \to \mathcal{C}$ for $R$ a local ring with principal maximal ideal. 
\end{remark}


\section{Modified link invariants} \label{sec:link}

We define link invariants for objects in $\C_R$. These can be normalized according to the nullity of the objects and yield nontrivial link invariants for objects in  $\C$ even if their categorial dimension is zero.

Let $\C$ be a ribbon category.
Then  one obtains for each labeling of
components of a link by objects of $\C$ an invariant of that link,
see e.g. Turaev's book \cite{Turaev}.
For our purposes it will be enough to do this via braids as follows:

By Markov's theorem, any link $L$ with $m$ components can be obtained
as the closure of a braid $\beta$ whose image in the canonical quotient map
into $S_n$ would be a permutation with $m$ cycles. Choose objects $X_i$,
$1\leq i\leq m$ in $\C$. We label the strands of $\beta$ which corresponds
to the $i$-th cycle by the object $X_i$. After conjugating by a suitable braid,
if necessary, we can assume that the first $c_1$ strands are labeled by $X_1$,
the next $c_2$ strands by $X_2$ etc, where $c_i$ is the number of strands
labeled by $X_i$. Using the braiding morphisms in $\C$
we obtain a linear map 
$$\Phi(\beta)\in \End(X_1^{\otimes c_1}\otimes X_2^{\otimes c_2}\otimes\ ...
\otimes X_m^{\otimes c_m}).$$
The link invariant $\L^{(X_1,\ ..., X_m)}(L)$ is then defined by
$$\L^{(X_1,\ ..., X_m)}(L)=Tr(\Phi(\beta)).$$
We now assume that our ribbon category $\C$
is defined over a local ring $R$ whose maximal ideal $(p)$ is generated
by an element $p$ in $R$, as in the previous subsection.

Then we can similarly also define the link invariant
$\L^{(X_1,\ ..., X_m), (k)}$ over the category $\C$ as follows: 
Let
$$\X^{\otimes \bf m}=X_1^{\otimes c_1}\otimes X_2^{\otimes c_2}\otimes\ ...
\otimes X_m^{\otimes c_m}$$
Then we have

\begin{theorem} \label{thm:mod-link} (a) If the object 
$\X^{\otimes \bf m}$
 is $k$-negligible, then we obtain a new link invariant
$\L^{(X_1,\ ..., X_m), (k)}$ defined by
$$\L^{(X_1,\ ..., X_m), (k)} (L) 
 =\frac{1}{p^k} \L^{(X_1,\ ..., X_m)}(L)$$
 which is well-defined. In particular, we obtain
a well-defined invariant with values in $R/(p)$.

(b) Let $\C = Tilt(U_v(\mathfrak{g}), R)$ where $R$ is the completion of $\Co[v]_{(v-q)}$ and $p=v-q$. Then
$R/(p)\cong \Co$ and the value of the $R/(p)$-valued invariant
 is equal to $\frac{1}{k!}\frac{d^k}{dv^k}\L^{(X_1,\ ..., X_m)}(L)_{|v=q}$,
which is valid for its evaluation on any $m$-component link $L$.
\end{theorem}

\begin{proof} The value of $\L^{(X_1,\ ..., X_m), (k)}(L)$ is just a renormalization of the
value of $\L^{(X_1,\ ..., X_m)}(L)$ which does not depend on the particular presentations of
$L$ via a braid $\beta$. Hence it is a link invariant.  Note that all constructions of $\L$ can be performed over the subring $\Co[v]_{(v-q)}$. To prove the last statememt from part (b), just observe that
$\L^{(X_1,\ ..., X_m)}(L)=(v-q)^k\L^{(X_1,\ ..., X_m),(k))}(L)$.
One shows easily by induction on $k$ that the $k$-th derivative of $\L^{(X_1,\ ..., X_m)}(L)$
at $v=q$ is equal to $k!\L^{(X_1,\ ..., X_m),(k))}(L)(q)$.
\end{proof}

\begin{remark}\label{modifiedremarks}
1. Our modified trace depends on the choice of the generator $p$. If we choose
a different generator $p'$, it is of the form $p'=ap$ for an invertible element in $R$.
Then the modified dimensions with respect to $p$ and $p'$ differ by
the same element $a^k$ for all objects in $\I$.

2. Observe that if $X$ is a simple object in $N_k$, then the dimension
of $X^{\otimes 2}$ would be in $I^{2k}$. Nevertheless,
$X^{\otimes 2}$ usually is not in $N_{2k}$. We would expect that
the nullity of the object $\X^{\otimes \bf m}$ would just be
the maximum of the nullities of the objects $X_i$.

3. It would be interesting to define modified traces over local
rings whose maximal ideals need more than one generator.

4. We expect that these constructions can also be applied to the theory of logarithmic Hopf link invariants as in \cite[Section 3.1.3]{Creutzig-Gannon}.
\end{remark}




\section{Tilting modules in the modular case}\label{tilting-modular}

We recall some statements about tilting modules over a field $\Bbbk$ and a complete discrete valuation ring $R$.

\subsection{Weights} Let $\Bbbk$ be a field of characteristic $p > 0$ and $G$ a semisimple, simply connected algebraic group over $\Bbbk$. We denote by $Rep(G)$ the category of finite dimensional representations. We fix a maximal torus $T$ and a Borel subgroup $B$. We denote by $R$ the set of roots and by $R^+$ the set of positive roots. The dominant integral weights are \[ X(T)^+ = \{ \lambda \in X(T) \ | \ <\lambda, \alpha^{\vee}> \geq 0\quad \forall \alpha \in R^+ \}.\]

\subsection{Induced modules and Weyl modules} The induction functor from $B$ to $G$ is not exact. For any $B$-module $M$ we put \[ H^i(M) = R^i Ind_B^G M, \ i \in \mathbb{N} \] and also abbreviate \[ H^i(\lambda) = H^i (k_{\lambda})\] where $\Bbbk_{\lambda}$ stands for $\Bbbk$ regarded as a $B$-module via $\lambda \in X(T)$. All these $H^i(\lambda)$ are finite dimensional over $\Bbbk$ \cite[I.5.12.c]{Jantzen}. The following properties are well-known:

\begin{enumerate}
\item For $\lambda \in X(T)^+$ we have $H^0(\lambda) \neq 0$.
\item If $\lambda \in X(T)^+$, then $soc(H^0(\lambda)) =: L(\lambda)$ is simple, and any finite dimensional simple $G$-module is isomorphic to exactly one $L(\lambda)$.
\item $L(\lambda)^* \cong L(-w_0 \lambda)$ where $w_0$ is the longest element of the Weyl group $W$.
\item $End_G(H^0(\lambda)) \cong \Bbbk \cong End_G L(\lambda)$.
\item $H^i(\lambda) = 0$ for $i > 0$ and $\lambda \in X(T)^+$ (Kempf vanishing).
\end{enumerate}

We define $V(\lambda) := H^0(-w_0 \lambda)^*$. We call $H^0(\lambda)$ the induced module and $V(\lambda)$ the Weyl module. A $G$-module $V$ has a good filtration if there is an ascending chain \[ 0 = V_0 \subset V_1 \subset \ldots \] such that $V = \bigcup_i V_i$ and $V_i/V_{i-1} \cong H^0(\lambda_i)$ for some $\lambda_i \in X(T)^+$. Likewise we say it has a Weyl filtration if $V_i/V_{i-1} \cong V(\lambda_i)$ for some $\lambda_i \in X(T)^+$. Then

\begin{enumerate}
\item $V$ admits a good filtration iff $Ext^1_G(V(\lambda),V) = 0 \ \forall \lambda \in X(T)^+$.
\item $V$ admits a Weyl filtration iff $Ext^1_G(V,H^0(\lambda)) = 0 \ \forall \lambda \in X(T)^+$.
\item $V$ has a Weyl filtration iff $V^*$ has a good filtration.
\item The tensor product of two modules with a good filtration has again a good filtration.
\end{enumerate}

\subsection{Tilting modules over $\Bbbk$}

A finite dimensional $G$-module is called a tilting module if it has a Weyl filtration and a good filtration. We denote by $Tilt(G,k)$ the full subcategory of tilting modules where $\Bbbk$ is a field of characteristic $p > 0$. The direct sum and the tensor product of two tilting modules is tilting. If $V_1,V_2$ are tilting modules, then $Ext^1_G(V_1,V_2) = 0 \ \forall i> 0$. 

\begin{proposition} \cite[Lemma E.6]{Jantzen} \label{prop:tilting-ind} For each $\lambda \in X^+$ there exists a unique indecomposable tilting module $T(\lambda)$ such that the weight space  $T(\lambda)_{\lambda}$ is free of rank 1 over $\Bbbk$ and such that $T(\lambda)_{\mu} \neq 0$ implies $\mu \leq \lambda$. Every tilting module can be written in a unique way as a direct sum of these $T(\lambda)$.
\end{proposition}


\subsection{Tilting modules over discrete valuation rings}

The entire theory of tilting modules can be developed over more general rings. We assume now that $R$ is a complete discrete valuation ring.

For a split reductive group $G$ over $\Z$ we denote by $G_R$ the group over $R$ obtained by extension of scalars to $R$ (but we may omit the subscript if there is no risk of confusion). As for $R = k$, one defines \[ H^i_R(M) = R^i Ind_{B_R}^{G_R} (M), \ H^i_R(\lambda) = H^i_R(R_{\lambda}).\] We denote by $V(\lambda)_{\Z}$ the $\Z$-form of the irreducible $G_{\Q}$-module $V(\lambda)$, $\lambda \in X(T)^+$ as in \cite[II.8.3]{Jantzen} and put \[ V(\lambda)_R = V(\lambda)_{\Z} \otimes_{\Z} R.\] By \cite[II.8.8, II.8.9]{Jantzen} $H^0_R(\lambda) \cong H^0_{\Z}(\lambda) \otimes_{\Z} R$, $H^i(\lambda) = 0$ for all $i>0$ and $V(\lambda)_R \cong H^0_R (-w_0\lambda)^{\vee}$. The analogue of Kempf's vanishing theorem permits to develop the theory of good filtrations and Weyl filtrations also over $R$. 

\begin{lemma} \label{lem:tilting} \cite[Lemma B.9, B.10]{Jantzen} Suppose $R$ is a complete discrete valuation ring. Let $M$ be a $G$-module that is free of finite rank over $R$. Then \ref{GF-1} - \ref{GF-3} are equivalent and \ref{GF-4} - \ref{GF-6} are equivalent.
\begin{enumerate}[label=\textbf{GF1.\arabic*}]
\item \label{GF-1} $M$ has a good filtration.
\item \label{GF-2} $Ext^1_G(V(\lambda)_R,M) = 0 \ \forall \lambda \in X(T)^+$.
\item \label{GF-3} For each maximal $\mathfrak{m}$ in $R$ the $G_{R/\mathfrak{m}}$-module $M \otimes R/\mathfrak{m}$ has a good filtration.
\item \label{GF-4} $M$ has a Weyl filtration.
\item \label{GF-5} $Ext^1_G(M,H^0_R(\lambda)) = 0 \ \forall \lambda \in X(T)^+$.
\item \label{GF-6} For each maximal $\mathfrak{m}$ in $R$ the $G_{R/\mathfrak{m}}$-module $M \otimes R/\mathfrak{m}$ has a Weyl filtration.
\end{enumerate}
\end{lemma}

Note that if a $G$-module $M$ has a good filtration, then $M$ is free over $R$ (since all $H^0(\mu)$ are so) and for each $R$-algebra $R'$ the $G_{R'}$-module $M \otimes R'$ has a good filtration since $H^0(\mu) \otimes R' \cong H_{R'}^0(\mu)$ \cite[8.8(1)]{Jantzen}. Furthermore $V(\lambda) \otimes R/\mathfrak{m} \cong V(\lambda)_{R/\mathfrak{m}}$ \cite[8.3]{Jantzen}. 

\medskip
In particular we can define tilting modules for $G_R$. Any tilting module for $G_R$ is free of finite rank over $R$. 

\begin{lemma} \cite[Lemma E.19, Proposition E.22]{Jantzen} Suppose that $R$ is a complete discrete valuation ring. For each $\lambda \in X^+$ there exists a unique indecomposable tilting module $T_R(\lambda)$ such that the weight space $T_R(\lambda)_{\lambda}$ is free of rank 1 over $R$ and such that $T_R(\mu)_{(\mu)} \neq 0$ implies $\mu \leq \lambda$. Every tilting module can be written in a unique way as a direct sum of these $T_R(\lambda)$.
\end{lemma}

\begin{remark} The existence of the $T(\lambda)$ with the correct properties works in greater generality (for instance if $R$ is a Dedekind ring which is a principal ideal domain, see \cite[Section E]{Jantzen}). While any tilting module for such $R$ can be decomposed into a direct sum of indecomposable tilting modules, their decomposition will in general not be unique.
\end{remark}

\subsection{Extension of scalars for tilting modules}

By \cite[E.20]{Jantzen} $M$ is an indecomposable $G_R$-module if and only if $End_{G_R}(M)$ is a local ring. We have \[ End_{G_R}(T(\lambda)) = R id_{T_R(\lambda)} + rad \ End_{T_R(\lambda)} (T_R(\lambda)), \]  where \[ rad \ End_{G_R} (T_R(\lambda)) = \{ \varphi \in End_{G_R} (T_R(\lambda)) \ | \ \varphi((T_R(\lambda)_{\lambda}) \subset \mathfrak{m} (T_R(\lambda))_{\lambda}  \}. \]
Furthermore \[ End_{G_{R/\mathfrak{m}}} (T_R(\lambda)) \otimes_R R/\mathfrak{m}) \cong End_{G_R} (T_R(\lambda)) \otimes_R R/\mathfrak{m}.\] From this one concludes that $End_{G_{R/\mathfrak{m}}}(T_R(\lambda)) \otimes_R R/\mathfrak{m})$ is a local ring and that therefore $T_R(\lambda) \otimes R/\mathfrak{m}$ is indecomposable.

\begin{corollary} \cite[E.20]{Jantzen} We have \[T_R(\lambda) \otimes R/\mathfrak{m}  \cong T_{R/\mathfrak{m}} (\lambda).\]
\end{corollary}






\section{Tilting modules for quantum groups}\label{tilting-quantum}

We review some results about quantized tilting modules which are needed to show that the category of tilting modules can be obtained as a mod $\mathfrak{m}$ evaluation.

\subsection{Lusztig's integral form} We denote by $\mathcal{A} = \Z[v,v^{-1}]$ the Laurent polynomial ring in an indeterminate $v$  and by $U = U_{\mathcal{A}}$ Lusztig's integral form (with divided powers) of the Drinfeld quantized enveloping algebra $U_v$ over $\Q(v)$ \cite{Lusztig-book}. Any commutative ring $R$ with 1 and a fixed invertible element $\bf{v}$ can be regarded as a commutative $\mathcal{A}$-algebra via the homomorphism $\phi: \mathcal{A} \to R$ such that $\phi(v^n) = {\mathbf{v}}^n$ for all $n \in \Z$. For any $\mathcal{A}$-algebra $R$ we denote $U_R = U_{\mathcal{A}} \otimes_{\mathcal{A}} R$. We apply this for $R = \Q(q)$ and specialize the generic parameter $v$ to a primitive $\ell$-th root of unity $q$ with $\ell$ odd, and assume that $\ell$ is not divisible by $3$ if $\mathfrak{g} \cong \mathfrak{g}_2$. We also assume $\ell > h$, the Coxeter number of $\mathfrak{g}$. Then we denote by $Rep(U_q(\mathfrak{g}))$ the category of finite dimensional representations of $U_q(\mathfrak{g})$ of type $1$ as in \cite{Andersen-Polo-Wen}.

\subsection{Quantized tilting modules over $\Q(q)$}

In $Rep(U_q(\mathfrak{g}))$ we have analogs of Weyl and induced modules (where the latter can be defined by means of the triangular decomposition of $U_q(\mathfrak{g})$) which we denote again by $V(\lambda)$ and $H^0(\lambda)$, $\lambda \in X^+$. These satisfy exactly the same $Ext$-vanishing conditions as in the modular case. In particular \[ Ext^i (V(\lambda), H^0(\lambda)) = 0\] for $i > 0, \ \lambda \in X^+$ (see \cite{Andersen-quantized} and the references therein). This uses the quantum variant of Kempf's vanishing theorem. We can then define modules with a good filtration and a Weyl filtration and define tilting modules as in the modular case. The main statement (proposition \ref{prop:tilting-ind}) about indecomposable tilting modules still holds and we denote by $T(\lambda)$ the indecomposable tilting module attached to $\lambda \in X^+$.

\subsection{Quantized tilting modules over $R$}

As in the modular case, the entire theory of tilting modules can be developed over a (complete) discrete valuation ring. The crucial ingredient here is that Kempf's vanishing theorem holds over the ground ring $\mathcal{A}$ as well. This follows from work of Ryom-Hansen \cite{Ryom-Hansen} and Kaneda \cite{Kaneda-cohom} \cite{Kaneda-cohom-preprint}. For an $\mathcal{A}$-algebra $R$ we denote by $H^0_R(\lambda)$ and $V_R(\lambda)$ the corresponding induced module and Weyl module. For any base change $\mathcal{A} \to R$ we have \[ H^0_{\mathcal{A}}(\lambda) \otimes_{\mathcal{A}} R = H^0_R (\lambda) \ \forall \lambda \in X^+ \] and likewise for $V_R(\lambda)$. Again by Kempf's vanishing theorem over $\mathcal{A}$ and standard facts about $H^i$ we get the vanishing of $Ext^i_{U_{\mathcal{A}}}(\mathcal{A}, H^0_{\mathcal{A}}(\lambda) \otimes H^0_{\mathcal{A}}(\mu) )$ for all $\lambda,\mu \in X^+$ and $i > 0$ (see also \cite{Kaneda-cohom-preprint}). The remaining arguments are the same as in the modular case. 
 In particular for each $\lambda \in X^+$ there is an indecomposable tilting module $T_R(\lambda)$, $\lambda \in X^+$, such that its $\lambda$-weight space is free of rank 1 \cite[Theorem 7.6]{Kaneda-cohom-preprint} and every tilting module decomposes uniquely into a direct sum of the $T_R(\lambda)$. 
 
Since the $\lambda$-weight space is free of rank 1, the arguments of Jantzen \cite[E.20]{Jantzen} go through and we obtain

\begin{corollary} Let $R$ be a complete discrete valuation ring with residue field $R/\mathfrak{m} = \mathbb{Q}(q)$ (or $\Co$) of characteristic 0. 
Then \begin{align*} T_R(\lambda) \otimes R/\mathfrak{m}  & \cong T_{R/\mathfrak{m}} (\lambda) \\  End (T_R(\lambda)) \otimes_R R/\mathfrak{m}) & \cong End (T_R(\lambda)) \otimes_R R/\mathfrak{m}.\end{align*}
\end{corollary}



\section{$k$-negligible ideals for tilting modules} \label{sec:tilting-eval}

The remaining necessary properties to define the $k$-negligible ideals for either quantized or modular tilting modules follow from Lusztig's theory of canonical bases \cite[Part 4]{Lusztig-book}.

\subsection{Weights and Weyl groups}

We will use the following notation (similar to \cite{Parshall-Wang}) in the following:

\begin{itemize}
\item $R$ an irreducible root system in an euclidian space $E$, $R^+$ a fixed set of positive roots, $\phi$ the simple roots.
\item $h$ the Coxeter number of $R$, $\rho$ the half sum of the positive roots.
\item $\alpha^{\vee} = 2\alpha/(\alpha,\alpha)$ the dual root to $\alpha \in R$.
\item $W$ the finite Weyl group.
\item $X$ the integral weight lattice of $R$.
\item $X^+$ the dominant integral weights.
\item $Q$ the $\Z$-span of $R$.
\item $W_{\ell} = W \rtimes T_{\ell}$ for the group of translations $T_{\ell} \ni t_{\ell}:E \to E$, $x \mapsto x + \ell\lambda$ for $\lambda \in Q$ and an odd integer $\ell$.
\item $d_{\alpha} = (\alpha,\alpha)/(\alpha_0,\alpha_0)$ for the shortest root $\alpha_0$ of $\phi$.
\item $wht(\lambda) = \sum_{\alpha \in \phi} r_{\alpha} d_{\alpha}$ for $\lambda = \sum_{\alpha \in \phi} r_{\alpha} \alpha$. 
\item $C_{\ell}$ the fundamental alcove $\{ \lambda \in X^+ \ | \ 0 < \ <\lambda + \rho, \alpha^{\vee}>  \ < \ell \ \text{ for all } \alpha \in R^+\}$ and $\overline{C}_{\ell}$ its closure. 
\end{itemize}

Recall that the affine Weyl group acts on $X$. The fundamental domain for this action is the closed alcove $\overline{C}_{\ell}$. The affine Weyl group $W_{\ell}$ can be identified with the set of alcoves in $X$ by matching $w \in W_{\ell}$ with $w \cdot C_{\ell}$. The alcove corresponding to $w$ is denoted by $C_w$. We denote $W_{\ell}^+ = \{ w \in W \ | \ w \cdot C_{\ell} \subset X^+ \} \subset W$.

\subsection{Based modules} \label{sec:based}

For the notion of a based module $(M,B)$ ($M \in \mathcal{C})$ we refer to  \cite[27.1.2]{Lusztig-book}. Based modules form a category $\tilde{C}$ with morphisms as in \cite[27.1.3]{Lusztig-book}. Based modules are closed under direct summands \cite[27.1.2]{Lusztig-book}, submodules and quotients \cite[27.1.3]{Lusztig-book}. 

\medskip
The tensor product of two based modules $M \otimes M'$ with the naive basis $B \otimes B'$ is not a based module, but there is a modified basis $B \lozenge B'$ with elements $b \lozenge b'$ with $(b,b') \in B \times B'$ which turns $(M \otimes M', B \lozenge  B')$ into a based module \cite[27.3]{Lusztig-book}. We denote by $_{\mathcal{A}}M \otimes M'$ the $\mathcal{A}$-submodule generated by the basis $B \otimes B'$ (where $\mathcal{A} = \Z[v,v^{-1}]$). The modified basis is then an $\mathcal{A}$-basis of $_{\mathcal{A}}M \otimes M'$ \cite[Theorem 27.3.3]{Lusztig-book}.

\medskip
The anti-involution $\omega$ of $U$ gives rise to a duality on $\mathcal{C}$, and we denote its dual by $^{\omega}M$. For a based module $(M,B)$ there is a partition of $B$ into subsets $B[\lambda]$ such that $B[0]$ is the base for the space of coinvariants $M_*$. For based modules $M,M'$ we can canonically identify $Hom(M,M')$ with the dual of the space of coinvariants of $M \otimes ^{\omega}M'$. This gives $Hom(M,M')$ the structure of a based module with basis $(B \lozenge ^{\omega}B')[0]$ \cite[27.2.5]{Lusztig-book}. In particular the $\mathcal{A}$-integrality properties of based modules show that $\tilde{C}$ is an $\mathcal{A}$-linear monoidal category.

The simple objects in $\tilde{C}$ are parametrized by $X^+$, i.e. every simple object is isomorphic to $(L(\lambda), B(\lambda),\Psi_{\lambda})$ for some unique $\lambda \in X^+$. Therefore any nontrivial based module admits a filtration with subquotients isomorphic to $(L(\lambda), B(\lambda),\Psi_{\lambda})$ for some $\lambda \in X^+$.

\begin{lemma} \label{lem:good-filtration-tensor} \cite[Corollary 1.9]{Kaneda-based} Let $M, M'$ be two based modules. Then $M_{\mathcal{A}} \otimes_{\mathcal{A}} M'_{\mathcal{A}}$ admits a filtration of $U_{\mathcal{A}}$-modules with each subquotient isomorphic to some $L_{\mathcal{A}}(\lambda), \ \lambda \in X^+$.
\end{lemma}

We can now base change from $\mathcal{A}$ to $R = \widehat{\Q[v]_{(v-q)}}$. Then the lemma immediately implies that the tensor product of two modules with a Weyl filtration (or good filtration) has a again a Weyl filtration (good filtration). 

Let $G_k$ as in section \ref{tilting-modular} denote a semisimple simply connected algebraic group over a field $\Bbbk$ of characteristic $p$. Then (as in \cite[1.10]{Kaneda-based}) we can view $\Z$ as an $\mathcal{A}$-algebra and obtain a ring isomorphism \[ U_{\mathcal{A}}/(K_{\alpha} - 1)_{\alpha \in \Pi} \ \otimes_{\mathcal{A}} \Z \cong Dist(G) \] for the Chevalley $\Z$-form $G$ of $G_k$. The $Dist(G)$-modules that are free of finite rank over $\Z$ are $G$-modules such that $L_{\mathcal{A}}(\lambda) \otimes_{\mathcal{A}} \Z$ is the Weyl module for $G$ of highest weight $\lambda, \ \lambda \in X^+$. Therefore lemma \ref{lem:good-filtration-tensor} implies in the modular case as well that the tensor product of two modules with a Weyl filtration (or good filtration) has a again a Weyl filtration (good filtration). 

\begin{corollary} The tensor product of two quantized tilting modules over $R$ and the tensor product of two modular tilting modules over the Witt ring $W(k)$ is a tilting module.
\end{corollary}

\subsection{The mod $\mathfrak{m}$ evaluation}

We want to realize now $Tilt(U_{q}(\mathfrak{g}), \Q(q))$ and $Tilt(G,k)$ as mod $\mathfrak{m}$ evaluations. In the modular case (by section \ref{tilting-modular}) we should take a complete discrete valuation ring of characteristic zero with residue field $\Bbbk$ of characteristic $p$. In order to measure the $p$-divisibility, its maximal ideal should be generated by $p$. By \cite[II.Theorem 3]{Serre-local} for every perfect field $\Bbbk$ of characteristic $p$, there exists a unique complete discrete valuation ring which is absolutely unramified [i.e., $p$ is a uniformizing element] and has $\Bbbk$ as its residue field: namely $W(k)$, the ring of Witt vectors. The results of section \ref{tilting-modular}, \ref{tilting-quantum} and \ref{sec:based} imply

\begin{theorem} \label{thm:tilting-eval} a) Let $R = \widehat{\Q[v]}_{v-q}$ where $q$ is a primitive $\ell$-th root of unity where $\ell > h$ and $\ell$ is not divisible by 3 if $\mathfrak{g}$ contains $\mathfrak{g}_2$. Then $Tilt(U_q(\mathfrak{g}),\Q(q))$ is the mod $\mathfrak{m}$ evaluation of $Tilt(U_v(\mathfrak{g}), R)$ where $\mathfrak{m} = (v-q)$.\\
b) Let $\Bbbk$ be a perfect field of characteristic $p$. Then $Tilt(G_{k},k))$ is the mod $\mathfrak{m}$ evaluation of $Tilt(G_{W(k)}, W(k))$ where $\mathfrak{m} = (p)$. 
\end{theorem}

As both rings are discrete valuation rings, all the $N_I$ are of the form $N_k$ for $k=1,2,\ldots$.

\subsection{The nullity in the quantum case}

The category $Tilt(U_q(\mathfrak{g}),\Q(q))$ is a braided monoidal category. The categorical dimension $\dim(V)$ is often called the quantum dimension of $V$. 

\begin{proposition} (\cite[Proposition 2.2, Theorem 3.3]{Parshall-Wang}) For $\lambda \in X^+$
\begin{enumerate}
\item We have \[ \dim V(\lambda) = \prod_{\alpha \in R^+} \frac{q^{-d_{\alpha}(\lambda + \rho,\alpha^{\vee})} - q^{d_{\alpha}(\lambda + \rho,\alpha^{\vee})}}{q^{-d_{\alpha}( \rho,\alpha^{\vee})} - q^{d_{\alpha}(\rho,\alpha^{\vee})}}.\]
\item $\dim V(\lambda) \neq 0$ if and only if $\lambda$ is $\ell$-regular..
\item If $w \in W_{\ell}$ satisfies $w \cdot \lambda \in X^+$, then \[ \dim V(w \cdot \lambda)  = (-1)^{l(w)} \dim V (\lambda)\] for the length function $l$ on $W$.
\end{enumerate}
\end{proposition}

The dimension formula can be rewritten as a product
\[ \prod_{\alpha \in R^+} \frac{[(\lambda +\ \rho, \alpha)]}{[(\rho, \alpha)]} \]
where $[ \ ]$ denotes the $q$ number. We deduce from this formula the following corollary.

\begin{corollary}\label{modulenullity}
 If $T(\la)=V(\la) $ is a simple tilting module, its nullity is equal to the number
$n_\ell(\la)$ of positive roots $\al$ such that $\ell |(\la+\rho,\al)$.
\end{corollary}

\begin{proof} If $T(\lambda)$ is simple, the nullity is determined by $Tr(id_{T(\lambda)})$.
\end{proof}

\begin{example} The Steinberg module has nullity $|R^+|$. Zeros occur in the dimension formula whenever $(\lambda + \rho, \beta)_{\ell}$ is divisible by $\ell$ for $\beta \in R^+$. For the weight of the Steinberg module this factor equals $(\ell \cdot \rho,\beta)$ and hence each factor for $\beta \in R^+$ is divisible by $\ell$. Since $<St>$ is the smallest thick ideal of $Tilt(U_q(\mathfrak{g}),\Q(q))$, any indecomposable module in $<St>$ has nullity $|R^+|$. 
\end{example}

\begin{example} In the $\mathfrak{sl}_2$-case the only nontrivial thick ideal is the ideal of negligible objects (the $N_1$ case). For $\mathfrak{sl}_3$ we have two nontrivial thick ideals (see \cite{Lusztig-homepage} for pictures of the thick ideals in the rank 2 cases), namely $N_1$ and $<St> = N_3$.
\end{example}


\subsection{The nullity in the modular case}

The category $Tilt(G,k)$ is a symmetric monoidal category. The categorical dimension $\dim(V)$ - also called the $p$-dimension - is the image of the vector space dimension $dim(V)$ under the homomorphism $\Z \to \mathbb{F}_p$, $1_{\Z} \mapsto 1_{\mathbb{F}_p}$.

\begin{proposition} (\cite[Example 1.6, Theorem 3.3]{Parshall-Wang}) For $\lambda \in X^+$
\begin{enumerate}
\item $\dim V(\lambda) \neq 0$ if and only if $\lambda$ is $p$-regular..
\item If $w \in W_p$ satisfies $w \cdot \lambda \in X^+$, then \[ \dim V(w \cdot \lambda)  = (-1)^{l(w)} \dim V (\lambda).\]
\end{enumerate}
\end{proposition}

If $T(\lambda)$ is irreducible, the nullity is already determined by $Tr(id_{T(\lambda)})$, i.e. the $p$-divisibility.

\begin{lemma} If $T(\lambda)$ is irreducible, then $T(\lambda)  \in N_k$ if and only if $p^k | dim T(\lambda)$. 
\end{lemma}

If $M$ is a module of $G$, we denote by $M^{[r]}$ the $G$-module which coincides with $M$ 
as a vector space on which the $G$-action is given by 
\[ g.m\ =\ F^r(g)m, \hskip 3em m\in M,\quad g\in G,\]
where the action on the right hand side is the original action of $G$ on $M$, and
$F$ is the Frobenius map, see \cite{Jantzen} for more details.

\begin{lemma}\label{dimmod}  If $T(\lambda)=V(\la)$ is a simple module of the algebraic group $G$
in characteristic $p$, then  the tilting module $T(\lar)$ is simple as well, 
where $\lar=p^r \lambda + (p^r - 1)\rho$.
\end{lemma}

$Proof.$  If $T(\lambda)$ is simple, it coincides with the Weyl module $V(\lambda)$ and the simple module $L(\lambda)$. A special case of Steinberg's tensor product theorem \cite[Proposition 3.16]{Jantzen} implies that

\[ L((p^r - 1)\rho) \otimes L(\lambda)^{[r]} \cong L((p^r - 1)\rho + p^r \lambda), \]
which also implies isomorphisms between the corresponding Weyl and tilting modules.
As the dimension of the Steinberg module $St_r=L((p^r - 1)\rho )=V((p^r - 1)\rho )$ is equal
to $p^{r|R^+|}$ (see e.g. the inductive formula in \cite{Jantzen}, Section 11.5), 
it follows  from the Weyl dimension formula that then also 
\[ \dim L((p^r - 1)\rho + p^r \lambda)=p^{r|R^+|}\dim L(\la)=\dim V((p^r - 1)\rho + p^r \lambda), \]
as $L(\la)=V(\la)$ by assumption. The claim follows from this, as 
$\dim  L(\lar)
= \dim  V(\lar)$.

\begin{corollary}
(a)  If $T(\la)=V(\la)$  is a simple tilting module, its nullity is equal to the number
$n_p(\la)=\sum_{\al>0} k_\al$, where the summation goes over all positive roots,
and $k_\al$ is the largest integer such that  $p^{k_\al} |(\la+\rho,\al)$.

(b) If $T(\la)=V(\la)$ is simple and has nullity $n(\la)$, then the nullity $n(\lar)$  of $T(\lar)$ is
equal to $n(\la)+r|R^+|$.
\end{corollary}


Clearly $T(\lambda) \in N_k$ always implies $p^k | \dim(T(\lambda))$. See section \ref{sec-modularA1} section \ref{sec:p-div} for examples that the converse doesn't hold.

\subsubsection{The Steinberg modules in the modular case} \label{steinberg-cells} 

Recall that we fixed a maximal torus $T$ and a Borel subgroup $T \subset B \subset G$. We denote by $F$ the Frobenius morphism. The notation $G_rT$ denotes the scheme $(F^r)^{-1}(T)$ and $G_r = ker(F^r)$.

The Steinberg module $St_r = L((p^r-1)\rho)$ \cite[II.3.18]{Jantzen} in $Rep(G)$ is both injective and projective as a $G_rT$ and $G_r$-module \cite[Proposition II.10.2]{Jantzen}. For the dimension of $St_r$ we obtain $\dim(St_r) = p^{rd}$, where $d$ is the number of positive roots. We consider for any $r \in \mathbb{N}$ the composition $\psi_r$ of monoidal functors \[ \xymatrix{Tilt(G,k) \ar[r] & Rep(G) \ar[r]^{res} & Rep(G_r T) \ar[r] & \overline{Rep(G_r T)} } \] where  $res: Rep(G) \to Rep(G_r T)$ denotes the restriction functor and $\overline{Rep(G_r T)}$ is the stable category of $Rep(G_r T)$. We denote the thick ideals that form the kernel of the functors $\psi_r$ by $I_r$ and the tensor ideals by $\mathcal{I}_r$. Clearly this gives an descending chain of thick ideals \[  I_1 \supseteq I_2 \supseteq I_3 \supseteq \ldots \] By \cite[Proposition 5.2.3]{Coulembier} and its proof the thick ideal $I_r$ is generated by $St_r$. The key ingredient is the following lemma of Jantzen \cite[Lemma II.E.8]{Jantzen}:


\begin{lemma} Let $\lambda \in X(T)^+$ and $r \in \mathbb{N}_{>0}$. Then $T(\lambda)$ is projective as a $G_r T$-module if and only if $<\lambda,\alpha^{\vee}> \geq p^r - 1$ for all simple roots $\alpha$.
\end{lemma}

Since $\dim St_r = p^{r |R^+|}$ and $St_r$ is simple, we obtain 

\begin{corollary} The nullity of $St_r$ is $r|R^+|$. In particular $<St_r> \subset N_{r|R^+|}$ .
\end{corollary}

\subsubsection{The modular cases $A_1$}\label{sec-modularA1}  For $SL(2)$ the $I_r$ are a complete list of thick ideals \cite[Theorem 5.3.1]{Coulembier}. A tilting module $T(m)$ is in $I_r$ if and only if $m \geq p^r - 1$. In particular $I_1 = N$. For $SL(2)$ we have $\dim(St_r) = p^{r}$. Therefore $I_r$ is the $r$-th negligible ideal, and every thick ideal is $k$-negligible for some $k$.

\medskip\noindent
It follows from Jensen \cite[Lemma 9.6]{Jensen-Phd} that for $p>2$ \[ T(\lambda) \in N_k \text{ if and only if } p^k | \dim T(\lambda) \] where $\dim$ refers to the dimension of $T(\lambda)$ as a vector space. In other words, $N_k$ measures the $p$-divisibility of the dimension of $T(\lambda)$.

\medskip\noindent
It is important to assume $p > 2$ here. Indeed the dimensions of the first tilting modules in the $p = 2$ case are \[ \dim T(0) = 1, \ \dim T(1) = St_1 = 2, \dim T(2) = 4, \ \dim T(3) = St_3 = 4.\]

Although $\dim T(2) = 4$, it is not in $N_2$. Over $\Z_2$ we have $Tr(id_{T(2)}) = 4$, but we can write $T(2) \cong T(1) \otimes T(1)$. Hence there is an endomorphism $f$ of $T(2)$ which permutes the two factors. It is easy to see that $Tr(f) = 2$, hence the trace is not always contained in $(2)^2$ and so $T(2) \notin N_2$.

\subsubsection{$p$-divisibility} \label{sec:p-div} The $SL(2)$-case seems to suggest that indeed we have $T(\lambda) \in N_k$ if and only if $p^k | \dim T(\lambda)$ provided $p > h$ (the Coxeter number). This is false. The following example was communicated to us by Thorge Jensen. For $Sp(4)$ and $p=11$ consider the following tilting modules (weights are expressed via fundamental weights)

\[
\begin{tabular}{|c|c|} $T(\lambda)$ & $\dim$ \\
\hline
$T(1,18)$ & $7018$ \\
$T(1,20)$ & $10648$ \\
$T(0,22)$ & $8954$ \\
$T(9,20)$ & $56870$
\end{tabular}
\]

By computations of Jensen these tilting modules all belong to the same $p$-cell (see Section \ref{sec:combi}), but their $p$-valuation is not constant. All these modules belong to $N_2$ but the dimension of $T(1,20)$ is divisible by $11^3$. Therefore the nullity of a $T(\lambda)$ is not simply given by the $p$-divisibility.


\section{Combinatorics for tilting modules} \label{sec:combi}

By theorem \ref{thm:tilting-eval} the $k$-negligible ideals are defined for modular and quantum tilting modules. Contrary to the case of Deligne categories, not every thick ideal is one of the $N_k$. We would like to understand the negligible ideals $N_k$ for modular and quantum tilting modules. In both cases the thick ideals are governed by the intricated Kazhdan-Lusztig cell theory of the affine Weyl group which is largely not understood in the modular case. In the following we try to give a more direct geometric description of these cells. While the results in the current section are general, we focus on type A in section \ref{sec:a-A}. 

\subsection{Classification of thick ideals} \label{sec:classif}

We first recall the classification of the thick ideals due to Ostrik and Achar-Hardesty-Riche. 

\subsubsection{Ostrik's classification} \label{quantum-cells}

\begin{definition} (\cite{Ostrik-ideals} \cite{Andersen-cells}) For $\lambda,\mu \in X^+$ write $\lambda \leq_q \mu$ if there exists $Q \in  Tilt(U_q(\mathfrak{g}),\Q(q))$ such that $T(\lambda)$ is a summand of $T(\mu) \otimes Q$. If both $\lambda \leq_q \mu$ and $\mu \leq_q \lambda$ write $\lambda \sim_q \mu$. The equivalence classes are called weight cells.
\end{definition}

We remark that for any $\lambda \in X^+$ we have $\lambda + \nu \leq_q \lambda$ for all $\nu \in X^+$ since $T(\lambda + \nu)$ is a summand of $T(\lambda) \otimes T(\nu)$. The fundamental alcove $C_{\ell}$ is the unique maximal weight cell in the $\leq_q$ ordering. For a weight cell $\underline{c}$ we denote by \[ T(\leq_q \underline{c}) \] the subcategory of objects in $Tilt(U_q(\mathfrak{g}),\Q(q))$ whose objects are direct sums of $T(\lambda)$ with $\lambda$ in a cell $\underline{c}'$ satisfying $c' \leq_q c$. Then $T(\leq_q \underline{c})$ is a thick ideal in $Tilt(U_q(\mathfrak{g}),\Q(q))$. 

\medskip\noindent 
The division of $X^+$ into weight cells gives a division of $W_{\ell}^+$. Recall that Lusztig and Xi have defined a partition of $W_{\ell}^+$ into right cells along with a right order $\leq_R$ on the set of right cells.

\begin{theorem} (\cite{Ostrik-ideals}) The weight cells in $X^+$ (and therefore the thick ideals in $Tilt(U_q(\mathfrak{g}),\Q(q))$) correspond to the right cells in $W^+_{\ell}$, i.e. for any right cell $A \in W_{\ell}^+$ the full subcategory $Tilt(U_q(\mathfrak{g}),\Q(q))_{\leq A} \subset Tilt(U_q(\mathfrak{g}),\Q(q))$ of direct sums of tilting modules $T(\lambda)$ for \[\lambda \in \bigcup_{w \in B \leq_R A} \underline{C}_w \] is a thick ideal.
\end{theorem}

By \cite[Remark 5.6]{Ostrik-ideals} every thick ideal is a sum of ideals of the form $Tilt(U_q(\mathfrak{g}),\Q(q))_{\leq A}$ for a right cell $A$ so that this theorem yields the classification of thick ideals in $Tilt(U_q(\mathfrak{g}),\Q(q))$. 


\subsubsection{Thick ideals in the modular case} \label{sec:modular-cells}

The notion of a weight cell can be defined in complete analogy to the quantum case in section \ref{quantum-cells}. We denote the modular analog of the preorder and the equivalence classes by $ \leq_T$ and $\sim_T$. An equivalence class is called a {\it modular weight cell} and $C_p$ is the largest modular weight cell. Contrary to the quantum case there are infinitely many modular weight cells (see section \ref{steinberg-cells}). Any modular weight cell $\underline{c}$ defines a thick ideal \[ T(\leq_T c).\] Ostrik's classification carries over to the modular case if we replace right cells by right $p$-cells (also called anti-spherical right $p$-cells).

\begin{theorem} (\cite[Theorem 7.7, Corollary 7.8]{Achar-Hardesty-Riche}) The modular weight cells in $X^+$ (and therefore the thick ideals in $Tilt(G,k)$) correspond to the right $p$-cells in $W^+_p$, i.e. for any right $p$-cell $A \in W_p^+$ the full subcategory $Tilt(G,k)_{\leq A} \subset Tilt(G,k)$ of direct sums of tilting modules $T(\lambda)$ for \[\lambda \in \bigcup_{w \in B \leq_{R,p} A} \underline{C}_w \] is a thick ideal.
\end{theorem}

As in the quantum case, every thick ideal is a sum of ideals attached to right $p$-cells.

\begin{example}
By \cite{Andersen-cells} the set \[ \underline{c}_1^r = (p^r-1)\rho + X_{p^r} + p^rC_p \] is a weight cell which contains $St_r$. We call this weight cell the $r$-th Steinberg cell. We have an equality of thick ideals $<St_r> = T(\leq \underline{c}_1^r)$.
\end{example}

\begin{remark} Since every thick ideal in the quantum and modular case is a sum of thick ideals attached to right cells, the nullity of a tilting module is constant on an alcove. Indeed, if $T(\lambda)$ is a tilting module in a thick ideal $I$ and $T(\lambda)$ is contained in an alcove $A$, the entire alcove is contained in $I$ since the ideal is a union of weight cells.
\end{remark}


\subsection{Affine Weyl groups}  
\def\ep{\epsilon}

We review some basic combinatorics in connection with affine Weyl groups and tilting modules.
See \cite{Jantzen}, \cite{JLink} , \cite{ALink}, \cite{Soergel1}, \cite{Soergel2} and the literature
quoted in these papers for more details. The use of facets for Kazhdan-Lusztig
combinatorics already appeared before e.g. in  \cite{Goodman-Wenzl}.

Let $X_n  $ be a finite root system, and let $X_n^{(1)}$
be the root system of the corresponding untwisted affine Weyl group. It has a faithful
representation in $\hstar$, where the generating reflections for $X_n$
act as usual, and the  additional generator acts via the reflection in the hyperplane
given by $(\theta, \gamma)=\ell$, where $\ell$ is a positive integer,
 $\theta$ is the long resp short root of greatest length
if $d|\ell$ resp $d\nmid \ell$; here $d$ is the ratio of the square length of a
long and a short root.

We obtain a system of hyperplanes on $\hstar$ from the orbits of the generating
hyperplanes under the affine Weyl group. They can be described
explicitly by 
$$H_{\alpha,k} \ =\ \{ x\in\hstar \ | \ (x,\alpha)=k\ell\},\quad \alpha\in \Delta_+, k\in\Z,$$
if $d|\ell$. If $d\nmid \ell$, we just replace the roots $\alpha$ by coroots in the definition
above.  The positive and negative
sides of these hyperplanes are defined in the obvious way, replacing the equal
sign in the definition by inequality signs.
These hyperplanes make $\hstar$ into a cell complex
as follows: We call an intersection of $k$ hyperplanes maximal if it has dimension
$n-k$, and we denote by $\h^*(n-k)$ the union of all maximal
intersections of $k$ hyperplanes. 
The set of $j$-cells then is given by all connected components of
$\h^*(j)\backslash \h^*(j-1)$, with $\h^*(-1)$ being the empty set.


\subsection{Alcoves and facets} 
As usual, we call the $n$-cells $alcoves$, and lower-dimensional cells $facets$.
The $(n-1)$-cells which are in the closure of a given alcove $A$ are called
the $walls$ of $A$.
The fundamental alcove $C_{\ell}$ is defined to be the unique alcove in the
dominant Weyl chamber whose closure contains the origin 0. We say that
a wall of $C_{\ell}$ corresponds to the simple reflection $s_i$ if it is fixed by it.
This defines a 1-1
correspondence between the walls of $C_{\ell}$ and the simple reflections.
We can now define a 1-1 correspondence between the alcoves 
and the elements $w$ of the affine Weyl group as well as a labeling of the walls via
simple reflections by induction on the length of $w$ as follows: 
The element 1 corresponds to $C_{\ell}$. If the alcove $A$ corresponds to
the element $w$, and $s_i$ is a simple reflection such that $ws_i$ has greater length
than $w$, then the alcove $A'=As_i$ corresponding to $ws_i$ is obtained by reflecting
$A$ in the wall labeled by $s_i$. This reflection also defines the labeling of the walls
of $A'$. Moreover, the action of the $s_i$ defines a right action of the affine Weyl
group on the alcoves.
The alcoves in the dominant Weyl chamber are in 1-1 correspondence with the
shortest elements of the cosets of the finite Weyl group in the
affine Weyl group. The Bruhat order then has the geometric interpretation
that 
$u<w$ is equivalent to the fact that whenever $u(C_{\ell})$ and $v(C_{\ell})$
lie on opposite sides of a hyperplane, $u(C_{\ell})$ must be
on the negative and $w(C_{\ell})$ must be on the positive side of that 
hyperplane. We similarly define for two facets $F_1$ and $F_2$ that $F_1<F_2$
if $F_2$ lies in or on the positive side of any hyperplane which contains $F_1$.

For a given facet $F$, the stabiliser group $W(F)$ is the group generated by the
reflections in the hyperplanes which contain $F$. We denote by $C_{\ell}(F)$ the
unique alcove whose closure contains $F$, and which is on the positive side of every hyperplane 
which contains $F$. The set $\Delta(F)$ denotes the positive roots  corresponding to the
hyperplanes which contain $F$ and a wall of $C_{\ell}(F)$. By definition, $C_{\ell}(F)$
is on the positive side of each of these hyperplanes. We call the reflections
corresponding to the roots in  $\Delta(F)$ the simple reflections of $W(F)$,
and the roots in $\Delta_F$ the simple roots of $W(F)$.
We also define the positive cone $C_+(F)$ to be the region
which is above all hyperplanes corresponding to the roots in $\Delta(F)$.


\subsection{Tilting modules and linkage}
If the context is not specified, the statement holds for
tilting modules of quantum groups at roots of unity as well as for
tilting modules of algebraic groups in characteristic $p$. Let $T(\la)$ be the unique indecomposable tilting module
up to isomorphism whose highest weight is $\lambda$. We will use the well-known
fact that if the Weyl module $V(\lambda)$ is simple,
it coincides with $T(\lambda)$ and with the simple module $L(\lambda)$ of highest weight 
$\lambda$.

\begin{theorem}\label{linkagetheorem} (Linkage Principle)
The Weyl module $V(\mu)$
appears in a filtration of the tilting module $T(\la)$ only if
$\mu$ is in the orbit of $\la$ under the affine Weyl group
and $\mu\leq \la$ in Bruhat order.
\end{theorem}


\subsection{Minimal facets}
The following lemma describes some tilting modules which are simple.
We call a facet $F$  $minimal$ if it lies in the $interior$ of the dominant Weyl chamber $C$
and no other facet in its orbit which also lies in the interior of $C$ can be smaller than it
in Bruhat order. We then have the following easy lemma: 

\begin{lemma}\label{minfacet} Let $F$ be a minimal facet and 
let $\la$ be an integral dominant weight such that $\la+\rho\in F$. 

(a) In both the quantum group case and in the modular case, we have
$T(\la)=V(\la)$ and the nullity of $T(\la)$
is equal to $k(F)$, the number of hyperplanes in which $F$ lies.

(b) Consider the modular case in
characteristic $p$. Let $\lar=p^r\la+(p^{r-1}-1)\rho$.
Then $T(\lar)=V(\lar)$ and the nullity of $T(\lar)$
is equal to $r|R^+|+k(F)$.
\end{lemma}

$Proof.$ By definition of minimal facet and the linkage principle, 
there exists no dominant integral weight
$\mu$ in the orbit of $\la$ such that $\mu<\la$ in Bruhat order.
Hence the only Weyl module which appears in the filtration of $T(\la)$ is
$V(\la)$ itself. The statement about the nullity is a direct consequence of
the dimension formula (for $p>h$  in the modular case).
The same argument also works for case (b), using Lemma \ref{dimmod}.


\subsection{Thick tensor ideals} 
Let $\I(F)$ be the thick ideal generated by the tilting modules $T(\la)$ with 
$\la+\rho\in F$. Recall that $C_+(F)$ is the region consisting of all
points $x\in \h^*$  which are on the positive side of any hyperplane which contains $F$.

\begin{proposition}\label{tensorexample} Let $F$ be a minimal facet,
and let $\la$ and $\lar$ be as in Lemma \ref{minfacet}. 

(a) In the quantum group case, the ideal $\I(F)$ contains
all modules  $T(\nu)$ with $\nu+\rho$ in $C_+(F)$. Any module in $\I(F)$ has nullity 
$\geq k(F)$.

(b) In the modular case, the ideal $\I(\Fr)$ generated by all
$T(\lar)$ with $\la+\rho\in F$ contains all modules $T(\nu)$ with $\nu\in\lar+X_+$,
where $X_+$ is the set of all dominant all dominant integral weights.
Any module in $\I(\Fr)$ has nullity at least $r|R^+|+k(F)$.
\end{proposition}

\begin{proof} Assume $\nu+\rho\in C_+(F)$. Then we can find a dominant integral weight $\la$
such that $\la+\rho\in F$ and $\nu-\la$ is a dominant weight.
Hence the tensor product $T(\la)\otimes T({\nu-\la})$ has highest weight $\nu$.
It follows that $T_\nu\in \I(F)$, and hence has at least the nullity of $F$.
\end{proof}

\begin{example} In the modular situation, the case $\lambda = 0$ corresponds to the Steinberg representations $St_r$.
\end{example}

In order to get a more concrete description of these tensor ideals,
it will be important to determine when two different facets generate
the same tensor ideal.

\begin{definition}\label{tensorequivalence} Let $F_1$ and $F_2$ 
be minimal facets. We say that they are tensor equivalent if $\I(F_1)=\I(F_2)$.
\end{definition}

\begin{lemma}\label{genequiv} Let $\I$ be the tensor ideal generated by the $simple$ object $T$,
and let $S$ be an object in $\I$ with the same nullity as $T$.
 Then also $S$ generates $\I$. In particular, if $F_1$ and $F_2$ are tensor equivalent minimal facets,
then $k(F_1)=k(F_2)$. 
\end{lemma}

\begin{proof} By assumption, there exists an object $W$ such that $S\subset T\otimes W$.
There exist maps 
$$\iota: \1\to W\otimes  W^*, \quad d: W\otimes  W^*\to \1$$
such that $d\circ\iota = \dim(W)$.
Moreover, by assumption, there exist morphisms
$\iota_S: S\to T\otimes W$ and $d_S: T\otimes W\to S$ such that
$d_S\circ\iota_S=id_S$.
Let $a\in\End(S)$ such that $Tr(a)$ has minimal nullity.
We then define maps $u: T\to S\otimes W^*$ and $v:  S\otimes  W^*
\to T$ by
$$u= (d_S\otimes id_{ W^*})\circ (id_T\otimes \iota),\quad v= 
(id_T\otimes d)\circ (\iota_Sa\otimes id_{W^*}).$$
It follows that $vu$ is an endomorphism of the simple
object  $T$, and hence
$vu=\alpha id_T$ for some scalar $\alpha$.
By functoriality of the trace operation, it follows from the definitions
$$\alpha \dim(T) = Tr(vu)=Tr(uv)=Tr(a).$$
As $Tr(a)$ and $\dim(T)$ have the same nullity, it follows that $\alpha$
is invertible. But then $\alpha^{-1}uv$ is an idempotent in
$\End(S\otimes \bar W)$ whose image is isomorphic to $T$.
\end{proof}

\begin{corollary} If $F_1>F_2$ and they have the same nullity,
then $\I(F_1)=\I(F_2)$ and $C_+(F_1)\subset C_+(F_2)$.
\end{corollary}

\begin{remark} \label{nullity-rem} 1.  In our cases, the nullity of the tensor ideal would correspond to
the length of the longest element of the stabilizer of $F$. On the other hand,
it is well-known that the longest element $w_0$ of a parabolic subgroup is in a cell for which the
$a$-function is equal to the length of  $w_0$.   For rank 2  and for type $A_n$,
these exhaust all two-sided cells.

2. All thick ideals are explicitly known in type $A$, see \cite{Shi}.  Each of them can
be associated to a parabolic subgroup, and hence to a facet.
However, there already seems to be a left cell 
for type $D_4$ whose value of the $a$-function, seven,  would not be the length
of the longest element of a parabolic subgroup of affine $D_4$, see \cite{DuD4},
\cite{ChenD4}.

\end{remark}

\def\lv{\vec l}



\section{Quantum and modular tilting modules in type A} \label{sec:typeA}

For $U_q(\mathfrak{sl}_n)$ two-sided cells of the affine Weyl group are parametrized by Young diagrams $\lambda$ of size $n$. We show that the thick ideal $\mathcal{I}(F_0(\lambda))$ agrees with the thick ideal attached to the two-sided cell by work of Shi and Ostrik. This also connects the nullity with the values of Lusztig's $a$-function. We propose a geometric description of the thick ideals for $U_q(\mathfrak{sl}_n)$ and for $SL(n)$. We assume throughout that $\ell$ and $p$ are bigger than the Coxeter number $h$, which is equal to $n$ in our case.

\subsection{Description of $\I(F)$s} We would like to get an elementary
geometric  description of the region in which the dominant integral weights
$\la$ lie for which $T(\la)$ is in $\I(F)$, for a given minimal facet.
For rank 2 these can be found in \cite{Lusztig-homepage}. It turns out 
that the regions corresponding to cells in the dominant Weyl chamber given by Lusztig 
coincide with the regions described in Proposition \ref{tensorexample}.
We expect that the regions described in Proposition \ref{tensorexample}
would also describe cells beyond rank 2.  We will illustrate this below
for some cases.

The cells for the affine Weyl groups of type $A$ have been 
determined by Shi in \cite{Shi}. Using Ostrik's results, this implies that
the thick ideals in type $A_{n-1}$ are
labeled by the partitions $\la$ of the set $\{ 1,2,\ ...,\ n\}$. 
As usual, we write partitions $\la=[\la_1,\la_2,\ ...\ \la_r]$
where the $\la_i$ are integers satisfying $\la_1\geq \la_2\geq\ ...\ \la_r >0$.
We identify them with Young diagrams, where $\la_i$ indicates the number
of boxes in the $i$-th row.
We denote weights of $\sl_n$ by the projections of $n$-tuples of integers
into the plane of $\R^n$ given by vectors whose coordinate sum is equal to 0.
We will usually only write the $n$-tuple of integers for simplicity of notation.
We will also use the notation $\lv$ for $\la+\rho$.

\begin{definition}\label{standardfacet}
Let $\ell$ be a positive integer, $\ell>n$, and let $s$ be the number of columns
of $\la$. Then we write the expression $(i-1)\ell+x_{j-1}$ into the box
of $\la$ in the $i$-th row and $j$-th column, where we have the convention
$0=x_0<x_1<\ ...\ < x_{s-1}<\ell$. The {\it standard facet $F_0(\la)$}
consists of all $n$-tuples whose coordinates consist of all possible
arrangements of numbers in the boxes of $\la$, written in descending order.
\end{definition}

\begin{definition} We call a facet $F$ strongly minimal if $F'\not < F$
for any $F'$ with $W(F')\cong W(F)$. Here all facets are assumed to be in the interior of $C$.
\end{definition}

\begin{definition}\label{simpleroot}
Let $F$ be a facet. We call $\al$ a root of the stabilizer $W(F)$ if 
there exists an integer $n_\al$ such that $(\al, x)=n_\al\ell$
for every $x\in F$. We call a root $\al$ positive if $(\al,x)>0$ for all $x\in F$.
Finally, we call a collection $R_F^+=\{ \al_i\}$ of positive roots of $W(F)$
a set of  simple roots of $F$
if every positive root of $W(F)$  can be uniquely written as a linear 
combination of the $\al_i$s with nonnegative integer coefficients.
\end{definition}

\begin{remark}\label{D(F)} The simple roots of $W(F)$ allow us an easy description
of the region $C_+(F)$: If $(\al, y)=n_\al\ell$ for all $y\in F$,
then $C_+(F)$ consists of all points $x\in\h^*$ such that 
 $(\al, x)\geq n_\al\ell$ for all simple $\al\in W(F)$. It is tempting to speculate
whether a description of the tensor ideal $\I(F)$ can be given via the region
$$D(F)\ = \bigcup_{W(F')=W(F)} C_+(F'),$$
where it would be enough to only consider strongly 
minimal facets $F'$ on the right hand side. We study this question in more
detail for type $A$ in the following sections.
\end{remark}

\subsection{More facets} Let $y$ be a point in the interior or lower closure of a given alcove.
The following definition will be useful for Shi's algorithm for identifying
cells which will be used later.

\begin{definition}\label{ychain}
 Let $y$ be a point in the interior or lower closure of a given alcove.
We call a subset $\{ y_i, i\in I\}$ of the coordinates of $y$
 a $y$-chain if $|y_i-y_j|\geq \ell$
for all $i,j\in I$. We call it an $\ell$-strict (or just strict) $y$-chain
if its elements are of the form $y_i-r\ell$ for $0\leq r < |I|-1$.
\end{definition}

Observe that if $y$ is a point in the standard facet $F_0(\la)$,
the coordinates of $y$ can be written as a disjoint union $\bigcup_j \Gamma_j$
of strict chains where $|\Gamma_j|=\la_j^T$. Namely $\Gamma_j$ consists of all the
coordinates of $y$ which are congruent to $x_{j-1}$ mod $\ell$,
where $x_0=0$. 

There is a more general procedure to produce facets
from any standard tableau of shape $\la$, i.e. for any filling of the
boxes of $\la$ such that the numbers increase along the rows and down
the columns. Let $i(r,c)$ be the number in the box in the $r$-th row and
$c$-th column. Then we define the facet $F_t$ by the points $y$ satisfying
the equalities
$$y_{i(r,c)}-y_{i(r+1,c)}=\ell$$
for any pair of boxes  $(r,c)$ and $(r+1,c)$ of our tableau. In general,
these equations may describe a collection of facets.
In this case, we pick the lowest one. It can be obtained by adding the
additional inequalities $y_i-y_{i+1}<\ell$, $1\leq i<n$. We now say that
\begin{equation}\label{faceconjugate}
 F\sim F_0(\la)
\end{equation}
if $F=F_t$ for some standard tableau of shape $\la$.
 We now also define the region $D(\la)$ by
\begin{equation}\label{Dlambdadef}
D(\la)\ =\ \bigcup_{F\sim F_0(\la)} C_+(F).
\end{equation}

\begin{remark} 1. Not every minimal facet is of the form $F_t$ for a standard
tableau. E.g. for $\sl_3$ the facet given by $y_2-y_3=\ell$ and $y_1-y_2<\ell$
is minimal, but can not be defined via a standard tableau $t$. However, we will
see that any strongly minimal facet can be obtained from a standard tableau.

2. If $1\leq i\leq n$, we define $\bar i=n+1-i$.
 If a root $\al$ is given by $\al(y)=y_i-y_j$,
we define the root $\bar\al$ by $\bar\al(y)=y_{\bar j}-y_{\bar i}$.
If $F$ is a facet defined by $\al(y)=\ell$ for certain roots $\al$,
we analogously define the facet $\bar F$ by the same equalities replacing
each $\al$ by $\bar\al$. We leave it to the reader to check that 
$F_0(\la)=\bar F_t$, where $t$ is the standard tableau obtained by filling
the boxes of $\la$ row by row.
\end{remark}

\subsection{Identifying minimal facets and their cells}
Shi has given several algorithms how to identify the 2-sided cell to which a given alcove
belongs. We review one of the here, and another one in the next section, following the
presentation in \cite{Cooper}, \cite{Cooper-journal}.

Given a point $y$ in the interior of the dominant Weyl chamber,
define a Young diagram $\mu=\mu(y)$ 
as follows: $\mu_1$ is the maximum of $|\Gamma_1|$,
for all possible $y$-chains $\Gamma_1$. Assuming we know $\mu_1$ up to $\mu_i$
we then define $\mu_{i+1}$ by the condition
\begin{equation}\label{mudef}
\sum_{j=1}^{i+1}\ \mu_j\ =\ max\{\ \sum_{j=1}^{i+1} |\Gamma_j|\ \},
\end{equation}
where the $\Gamma_j$s are mutually disjoint $y$-chains, and the maximum is taken
over all possible collection of $i+1$ disjoint $y$-chains.

\begin{proposition}\label{identifyideals} If $y$ is a point on a facet $F\sim F_0(\la)$ for
the Young diagram $\la$, we have $\mu(y)=\la^T$. Moreover, $F_0(\la)$ is a minimal facet.
\end{proposition}

\begin{proof}  Let $y$ be a  point on
the  facet $F$. Let $\Gamma_i$ be the chain consisting of all the  entries
$y_i$ which are congruent to $i-1$ mod $\ell$; by construction it has $\la_i^T$ elements.
The claim would follow
if we can show that for any other collection of mutually disjoint $y$-chains
$\tilde \Gamma_j$ we have 
$$\sum_{j=1}^{i+1} |\tilde \Gamma_j|\ \leq \ \sum_{j=1}^{i+1} |\Gamma_j|,
\quad {\rm for\ } 0\leq i<n.$$
Let $I(m)$ be  the number of indices $k$ for which
$m\ell\leq y_k<(m+1)\ell$. Then obviously the number of such indices in a disjoint
union of $i$ chains is less or equal to the minimum of $i$ and $I(m)$.
Summing over all $m$ from 0 to $n$ shows that the largest possible
number we can get is the number of boxes of $\la$ in its first $i$ columns.
As by induction assumption $\mu_j=\la_j^T$ for $1\leq j\leq i$, it follows that
$\mu(y)=\la^T$.

To prove the second claim, let $y$ be the point on $F_0(\la)$
with $x_j=j-1$. 
Observe that any point $\tilde y\in W.y\cup C$ with $\tilde y\leq y$
 needs to have the same $L^1$
norm, it needs to have the same number of coordinates congruent
$i$ mod $\ell$ as $y$ for each $i$, and its coordinates need
to be positive and strictly decreasing. By construction, the coordinates of $y$ are the smallest
possible numbers subject to these constraints. Hence if $\tilde y\neq y$,
we could find an $i$ such that $\sum_{j=1}^i \tilde y_j>\sum_{j=1}^i  y_j$,
 contradicting $\tilde y\leq y$.
\end{proof} 

\subsection{Description of $D(\la)$} We now describe a second way how to determine
the two-sided cell to which a given alcove in the dominant Weyl chamber belongs.
It is also due to Shi.
Given a point $y$ in the dominant Weyl chamber, we  construct a standard tableau
$t_y$
as follows: We start with putting the number 1 into the top-left box.
We then add the box with the number 2 on the right if $y_1-y_2<\ell$,
and we add it below the first box if $y_1-y_2\geq \ell$.
Having placed boxes with the numbers 1 until $i$, we add the box containing
$i+1$ at the bottom of the left-most column such that $y_r-y_{i+1}\geq \ell$,
where $r$ is the number in the lowest box of that column. Then one
can show (see e.g. \cite{Cooper}, Section 3.2 and 3.2):

\begin{proposition}\label{Dyalgorithm}
If $y$ is a dominant integral weight, the procedure above constructs a 
standard tableau $t_y$. If $\la$ is the associated Young diagram,
and $y$ is in the lower closure of the alcove $A$, then $A$ is in the two-sided
cell labeled by $\la^T$.
\end{proposition}

\begin{theorem}\label{Ilambda} The indecomposable module $T(\nu)$ is
in the ideal $\I(\la)$ generated by the facet $F_0(\la)$ if and only if
$\nu+\rho$ is in the region $D(\la)$ as defined in \ref{Dlambdadef}.
\end{theorem}

\begin{proof} It follows from Proposition \ref{identifyideals} that whenever
$\nu +\rho\in F$ with $F\sim F_0(\la)$, then $T(\nu)\in\I(\la)$.
Hence $T(\nu)\in\I(\la)$ whenever $\nu+\rho\in D(\lambda)$ by
Proposition \ref{tensorexample}. To prove the other inclusion,
let $T(\nu)$ be an indecomposable tilting module with
highest weight $\nu$ such that $\nu+\rho$
is in an alcove belonging to the two-sided cell labeled by $\la^T$.
Then we obtain a tableau $t_y$ of shape $\la$, where $y=\nu+\rho$.
Let $i(r,c)$ be the number in the box in the $r$-th row and $c$-th column.
Then by construction $y_{i(r,c)}-y_{i(r+1,c)}\geq \ell$.
Define $F$ as the set of points $x$ such that  $x_{i(r,c)}-x_{i(r+1,c)}= \ell$.
This is exactly the facet obtained from $t_y$, see the discussion
above \ref{faceconjugate}.
Then obviously $y\in C_+(F)$.
\end{proof}

\subsection{Examples}
In the following we list
 the  {\it strongly minimal} facets $F $ conjugate (as defined in
\ref{faceconjugate}) to $F_0(\la)$ for $\sl_n$, $n\leq 5$,
for each Young diagram $\la$ with $n$ boxes.
 We use the convention that $0<x_1<x_2<\ ...\ <\ell$.
The region $C_+(F)$ is then given by all points $y$ satisfying
$(y,\al)\geq \ell$ for all roots $\al$ listed under simple roots.
We start with $\sl_3$, where all strongly minimal facets are standard
facets:
\[
\begin{tabular}{|c|c|c|c|} $\la$ & facet & nullity & simple roots\\
\hline
$[3]$ & $(x_2,x_1,0)$ &0&$\emptyset$\\
$[2,1]$ & $(\ell, x_1, 0)$ &1&$\ep_1-\ep_3$\\ 
$[1^3]$&$ (2\ell, \ell, 0)$&3&$\ep_1-\ep_2,\ep_2-\ep_3$\\
\end{tabular}
\]
In the following, we list all strongly minimal facets for $\sl_4$.
Observe that we also have a strongly minimal facet
which is not a standard facet for the diagram $[2,1^2]$.

\[
\begin{tabular}{|c|c|c|c|} $\la$ & facet & nullity & simple roots \\
\hline
$[4]$ & $(x_3, x_2,x_1,0)$ &0&$\emptyset$\\
$[3,1]$ & $(\ell, x_2, x_1, 0)$ &1&$\ep_1-\ep_4$\\ 
$[2^2]$&$ (\ell+x_1, \ell, x_1, 0)$&2&$\ep_1-\ep_3, \ep_2-\ep_4$\\
$[2, 1^2]$&$ (2\ell, \ell, x_1, 0)$&3&$\ep_1-\ep_2,\ep_2-\ep_4$\\
&$ (2\ell, \ell+x_1, \ell,  0)$&3&$\ep_1-\ep_3,\ep_3-\ep_4$\\
$[1^4]$&$ (3\ell, 2\ell, \ell, 0)$&6&$\ep_1-\ep_2,\ep_2-\ep_3,\ep_3-\ep_4$\\
\end{tabular}
\]
For $\sl_5$, we will omit the facets corresponding to $[5]$
(the fundamental alcove), to $[4,1]$ (the wall of the fundamental
alcove which is not in a reflection plane of the finite Weyl group)
and $[1^5]$, which consists of the point $(4\ell, 3\ell, 2\ell, \ell, 0)$
corresponding to the Steinberg module.
The remaining facets then consist of

\[
\begin{tabular}{|c|c|c|c|} $\la$ & facet & nullity \\
\hline
$[3,2]$ & $(\ell+x_1, \ell, x_2,x_1,0)$ &2&$\ep_1-\ep_4, \ep_2-\ep_5$\\
$[3,1,1]$ & $(2\ell,\ell, x_2, x_1, 0)$ &3&$\ep_1-\ep_2,\ep_2-\ep_5$\\ 
 & $(2\ell,\ell+x_1, \ell, x_2,  0)$ &3&$\ep_1-\ep_3,\ep_3-\ep_5$\\
 & $(2\ell,\ell+x_2,\ell+x_1,\ell,  0)$ &3&$\ep_1-\ep_4,\ep_4-\ep_5$\\
$[2^2, 1]$&$ (2\ell, \ell+x_1, \ell, x_1, 0)$&4&$\ep_1-\ep_3,\ep_3-\ep_5, \ep_2-\ep_4$\\
$[2, 1^3]$&$ (3\ell,2\ell, \ell, x_1, 0)$&6&$\ep_1-\ep_2,\ep_2-\ep_3,\ep_3-\ep_5$\\
&$ (3\ell,2\ell, \ell+x_1, \ell,  0)$&6&$\ep_1-\ep_2,\ep_2-\ep_4, \ep_4-\ep_5$\\
&$ (3\ell,2\ell+x_1, 2\ell, \ell,  0)$&6&$\ep_1-\ep_3,\ep_3-\ep_4,\ep_4-\ep_5$\\
\end{tabular}
\]


\subsection{The $a$-function} \label{sec:a-A}

Recall from the introduction that Lusztig defined an integer valued function on two-sided cells. 
In type $A$ there is an easy description of the $a$-function \cite[Section 4.2, 5.2]{Elias-Hogancamp}. By \cite{Shi} the two-sided cells of the affine Weyl group $W$ are parametrized by partitions of $n$. It is well-known that it is sufficient (in type A) to compute the value of the $a$-function in the finite case $W_n = S_n$. By \cite{Elias-Hogancamp} \[ a(\lambda^T) = r(\lambda) \] where $r(\lambda)$ is the row number of $\lambda$, the sum over the row numbers of the boxes in the Young diagram where the row number
of a box in the $i$-th row is $i-1$ (note that we use $\la^T$ for the partition labeled by $\la$ in \cite{Elias-Hogancamp}).

\begin{example} For $n=8$ consider the partition $\lambda = (4,3,1)$. Its row number is $0 + 0 + 0  + 0 + 1 + 1 + 1 + 2 = 5$. Hence the value of the $a$-function on the two-sided cell associated to $\lambda$ is 5. 
\end{example}

\begin{theorem}\label{TypeAmain}
The thick ideal $\I(\la)=\I(F_0(\la))$
generated by  the tilting modules $T(\nu)$ for which
$\nu+\rho\in F_0(\la)$ coincides with the thick ideal constructed by
Ostrik for the cell in the dominant Weyl chamber corresponding to the two-sided
cell labeled by the Young diagram $\la^T$.  In particular, the nullity of any
generating module $T(\nu)$ of that ideal is equal to the $a$-function of that cell.
Moreover, that thick ideal
contains all tilting modules $T(\nu)$ for which $\nu+\rho\in D(\la)$, 
as defined below (\ref{faceconjugate}).
\end{theorem}

\begin{proof}
Let $T(\nu)$ be a tilting module for which
$y=\nu+\rho\in F_0(\la)$. As $F_0(\la)$ is a minimal facet
by Proposition \ref{identifyideals}, $T(\la)=V(\la)$.
Hence its nullity can be determined from its dimension, see
Corollary \ref{modulenullity}. By construction of $F_0(\la)$
$y_i$ is congruent $y_j$ mod $\ell$ if $y_i$ and $y_j$ are
in the same column of $\la$. Hence the nullity
is given by 
$\sum_i \binom{\la_i^T-1}{2}$, which is equal to the row
number  $r(\la)$. This shows that the nullity of $F_0(\la)$ coincides
with the $a$-function.

 We have seen in  Proposition \ref{identifyideals}  that if $y\in F$
with $F\sim F_0(\la)$
then $\mu(y)=\la^T$. Hence any module $T(\nu)$ with $\nu+\rho\in F$
is also in $\I(\la)$. The last statement now
follows from  Proposition \ref{tensorexample}.
\end{proof}

\begin{remark} \label{a-fun} It was shown by Ostrik for $\ell$ a large enough prime (see  \cite{Ostrik-dim} \cite{Ostrik-support})  
that the dimension of any tilting module corresponding to a two-sided cell $A$
 is divisible by $\ell^{a(A)}$ and for any cell there exists a tilting module such that its dimension is not divisible by $\ell^{a(A)+1}$.
Hence the relationship between nullity and $a$-function in Theorem \ref{TypeAmain} would hold whenever one can show that the nullity of such a module is given by
its dimension.
So it would seem reasonable to expect this to be true in general. But see the example in  Section \ref{sec-modularA1}
where the nullity of a modular tilting module is not given by its dimension.  
\end{remark}

\begin{remark}\label{Dlambda-conjecture} We were able to give a fairly explicit description of the tensor ideals
thanks to the work of Shi \cite{Shi}. In the formulation as $D(F)$, see Remark \ref{D(F)}, our approach
might also be useful to characterize tensor
ideals for other Lie types. E.g. this formulation works for all rank 2 cases, but not for all ideals for type $D_4$,
see Remark \ref{nullity-rem}.

\end{remark}

\begin{remark} The thick ideal $N_k$ is the sum of the $I(\lambda)$ $(\lambda$ partition of $n$) for which the nullity is $\geq k$.
\end{remark}



\subsection{Regions for the modular case} \label{sec:modular-regions} We can use the results from the
previous subsection also for the modular case; to  conform with the usual
notations, we replace $\ell$ by the prime $p$. In addition, we can use
the ``telescoping effect" in Lemma \ref{minfacet} (b) and
Proposition \ref{tensorexample} (b) to construct infinitely many
thick ideals. We will also extend our conjectural description of these ideals
and the related cells  to this case. We use the same notation
as in the previous subsection. Let $F_0(\la)$ be the standard minimal facet
associated to the Young diagram $\la$, and let $r$ be an integer, $r\geq 1$.
Then we construct the
region $D_r(\la)$ by $D_1(\la)=D(\la)$ and

\[ 
D_r(\la)\ =\ \bigcup_{F \sim F_0(\la)} (\bigcup_{\la\in F} \lar+X^+), \quad r>1
\]
where $\lar$ is as in Lemma \ref{minfacet} (and $\la$ is interpreted as a weight), and the first union goes 
over all strongly minimal facets $F$ which are equivalent to $F_0(\la)$. The following proposition
again follows from Lemma \ref{minfacet} and Proposition \ref{tensorexample}.

\begin{proposition} If $\mu+\rho$ is in $D_r(\la)$ its nullity is at least equal
to $r|R^+|+k(F_0(\la))$.
\end{proposition}

\begin{question} Is every thick ideal in type A a sum of ideals of the form $D_r(\lambda)$? 
\end{question}

An affirmative answer would give a geometric description of the right $p$-cells in the affine Weyl group.

\begin{remark}\label{modularguessremark} 1. It is easy to see for $r\geq 1$
that $D_r(\emptyset)=D_{r-1}([1^k])$. We will identify these regions in the following.

2. For type $A_1$ the thick ideals are generated by the Steinberg modules.
It is easy to see that our regions do describe the thick ideals in this case.
An explicit description of the tensor ideals has also been given for type $A_2$,
see section \ref{sec:a2}. One checks easily that this is compatible with our regions.
\end{remark}

\subsubsection{The modular $A_2$-case} \label{sec:a2}

By \cite[Example 15]{Andersen-cells} \cite[Chapter 10]{Jensen-Phd} the thick ideals are given by the the Steinberg cells $<St_r> = T(\leq \underline{c}_1^r)$ and the tilting modules associated to the weight cell \[ \underline{c}_2^r = Y_r \setminus (Y_{r+1} \cup \underline{c}_1^r)  \ .\] A beautiful picture illustrating the $p= 5$-case can be found in  \cite{Andersen-cells} .

In the $\mathfrak{sl_2}$-case we have $<St_r> = N_{r|R^+|}$, so in order to determine the $k$-negligible ideals it suffices to find one weight $\lambda$ in each $c_2^r$ such that the associated tilting module is a Weyl module and compute its nullity with the dimension formula. All in all we obtain

\begin{enumerate}
\item The Steinberg ideals with nullity $3s$ for $s=0,1,2,\ldots$ and
\item for $s=0,1,2,\ldots$ the ideals generated by the $T(\lambda)$ with \[ (\lambda+\rho, \rho)= p^{s+1}, \  (\lambda+\rho, \alpha_1)= rp^s, 1\leq r <p \] of nullity $3s+1$.
\end{enumerate}


\section{Questions}\label{sec:questions}

\subsection{Extension to more general categories} It would be interesting to extend the definiton of $k$-negligible morphisms or ideals to other monoidal categories or supercategories.

\subsection{$q$-Deligne categories at roots of unity} $k$-negligible ideals can be also defined for the $q$-versions of the Deligne categories in type $ABCD$. For generic $q$ the classification of tensor ideals and thick ideals for these categories is the same as for the classical Deligne categories \cite{Brundan-skein} \cite{Coulembier}, but is unknown for $q$ a root of unity.  

\subsection{Modified traces} Currently we define modified traces only if the maximal ideal has  one generator. It would be interesting to define modified traces if the maximal ideal is not principal. We also expect that our construction defines modified traces for other categories. 

\subsection{$a$-Function in the quantum and modular case} We expect an equality between the nullity and the $a$-function in the quantum case for $\ell$ large enough for all types (i.e. $N_k$ would correspond to the union of weight cells with $a$-value $\geq k$ for $\ell$ a large enough prime). There does not seem to be an accepted definition for the $a$-function in the modular case if the Kazhdan-Lusztig basis is replaced with the $p$-canonical basis. One might wonder if there is again a connection with the nullity.



\section*{Acknowledgements}

The work was started while the first author was a Postdoc at the MPI in Bonn and the second author a visiting scientist. We thank the MPI for the ideal working enviroment. We would also like to thank Henning Haar Andersen, Kevin Coulembier, Pavel Etingof, Jens Carsten Jantzen, Thorge Jensen, Hankyung Ko, Catharina Stroppel and Daniel Tubbenhauer for helpful discussions. We would like to thank the referee for a very helpful review. The research of T.H. was partially funded by the Deutsche Forschungsgemeinschaft (DFG, German Research
Foundation) under Germany's Excellence Strategy – EXC-2047/1 – 390685813.



\end{document}